\documentclass[10pt]{amsart}
\usepackage{amsmath,amssymb,url}
\usepackage{enumitem} 
\usepackage{kantlipsum}
\usepackage{microtype}
\usepackage{bm}
\usepackage{tikz}
\usepackage{mathrsfs}
\usepackage{mathabx}
\usepackage{comment}
\usepackage{color}
\usepackage[all]{xy}
\usepackage{relsize}

\numberwithin{equation}{section}

\theoremstyle{plain}
\newtheorem{theorem}{Theorem}[section]
\newtheorem{lemma}[theorem]{Lemma}
\newtheorem{proposition}[theorem]{Proposition}
\newtheorem{corollary}[theorem]{Corollary}

\theoremstyle{definition}
\newtheorem{definition}[theorem]{Definition}
\newtheorem{example}[theorem]{Example}

\newtheorem{remark}[theorem]{Remark}

\newcommand{\lltimes}{\lltimes\!}
\newcommand{\e}{\mathsf e}
\newcommand{\sq}{\mathsmaller{[]}}
\newcommand{\pref}{\!\mathsmaller{<}}

\makeatletter
\@namedef{subjclassname@2020}{\textup{2020} Mathematics Subject Classification}
\makeatother
\begin{document}

\title[From Clones to Cm-Monoids]{From Clones to Cm-Monoids} 
\author[A. Bucciarelli, P.-L. Curien and A. Salibra]{Antonio Bucciarelli, Pierre-Louis Curien and Antonino Salibra}
\address{Institut de Recherche en Informatique Fondamentale\\
CNRS and Universit\'e Paris Cit\'e\\ 8 Place Aur\'elie Nemours, 75205 Paris Cedex 13, France}

\begin{abstract}
  Clones of functions play a foundational role in both universal algebra and theoretical computer science. In this work, we introduce the variety of \emph{clone merge monoids} (cm-monoids), a unifying one-sorted algebraic framework that integrates abstract clones, clone algebras and Neumann's abstract $\aleph_0$-clones. Cm-monoids combine a monoid structure with a new algebraic structure called merge algebra, capturing essential properties of infinite sequences of
  functions.
We establish a categorical equivalence between clone algebras and 
finitely ranked cm-monoids.
This equivalence yields by restriction  a three-fold equivalence between abstract clones, finite-dimensional clone algebras, and finite-dimensional, finitely ranked cm-monoids,  and is itself obtained by restriction  from  a categorical equivalence  between \emph{partial infinitary clone algebras} (which generalise clone algebras and abstract $\aleph_0$-clones) and extensional cm-monoids.
\end{abstract}


\maketitle

\section*{Dedicated to Robert Goldblatt}\label{sec:intro}

The notion of Lambda Abstraction Algebra ($\mathsf{LAA}$), introduced by Pigozzi and Salibra in \cite{PS93,PS95}, provides an algebraic framework for modeling the essential features of the untyped lambda calculus. Unlike traditional approaches based on combinatory logic (see \cite{Bare84}), $\mathsf{LAA}$s offer a purely equational setting, where lambda abstraction is treated as a primitive operation governed by suitable identities. 

In \cite{GS99}  Goldblatt and Salibra  established a general representation theorem, namely that every $\mathsf{LAA}$ is isomorphic to an algebra of functions arising from the coordination of a model of the lambda calculus. This theorem provides a foundational result that enabled the resolution of several significant open problems in lambda calculus (e.g. see \cite{MS10}).

As shown in this paper, in each $\mathsf{LAA}$ $\mathbf A$ there exist term-operations making $\mathbf A$ into a clone algebra.
These algebras  have been recently introduced by Bucciarelli and  Salibra \cite{BS22} in order to provide a one-sorted algebraic account of clones of functions,
which  are sets of finitary functions that include all projections and are closed under composition (see \cite{L06,SZ86,T93}).
The main result of \cite{BS22} is a representation theorem stating that every clone algebra is isomorphic to a clone of functions from $A^\omega$ to $A$ for a suitable set $A$,  much in the spirit of the aforementioned representation theorem for $\mathsf{LAA}$.

\bigskip

Clones of functions play a significant role in universal algebra, as the set of all term operations of an algebra always forms a clone, and, in fact, every clone is of this form.
Therefore, comparing clones of algebras is much more appropriate than comparing their basic operations for the purpose of classifying algebras based on different behaviours. 
In addition to their relevance in universal algebra, clones also play an important role in the study of first-order structures. The polymorphism clone of a first-order structure, consisting of  all finitary functions that preserve the structure, holds valuable information and serves as a powerful analytical tool.
In particular, clones  have significant applications in theoretical computer science, especially in the context of constraint satisfaction problems (CSPs) (e.g. see \cite{BKW17,B21}). In a CSP, a specific structure (the template)  is fixed, and the problem involves deciding whether a given conjunction of atomic formulas over the signature of the template is satisfiable in that structure. Jeavons' groundbreaking discovery \cite{J98} revealed that the complexity of a CSP for a finite structure is entirely determined by the polymorphism clone of that structure. 

\medskip

Abstract clones were introduced by Philip Hall (see \cite{C81,evans82,T93} and \cite[Chapter 10]{FMMT22a}) to provide an axiomatic generalisation of clones, analogous to the way groups generalise permutation groups. An abstract clone is a many-sorted algebra $\mathbf A = (A_n, q^n_k, \e^n_k)$, where the sorts are indexed by the natural numbers, representing the arities of finitary functions. Projections are abstracted into an infinite system of nullary operations $\e^n_k \in A_n$ ($k < n$), while the many-sorted composition of finitary functions is captured by a family of operations $q^n_k: A_k \times (A_n)^k \to A_n$. Every abstract clone is isomorphic to a clone of functions.
Abstract clones also have a natural connection to category theory. In a Lawvere theory \cite{law63}, arities are treated as objects, and morphisms from $m$ to $n$ represent abstract $n$-tuples of functions of arity $m$. A classical result states that the categories of abstract clones and Lawvere theories are equivalent.
 Since Lawvere theories are, in turn, equivalent to finitary monads on the category of sets, it follows that the category of abstract clones is also equivalent to the category of finitary monads.

\medskip

While clones are inherently many-sorted structures, efforts have been made to encode them into one-sorted algebras to take full advantage of the machinery of universal algebra. The basic idea is to switch from $n$-ary functions for varying $n$ to $\omega$-ary ones.
This has led to the development of abstract $\aleph_0$-clones ($\aleph_0$-$\mathsf{AC}$s) \cite{neu70} and clone algebras ($\mathsf{CA}$s) \cite{BS22}, where projections are represented by a countably infinite system of nullary operations $\e_0, \e_1, \dots, \e_n, \dots$, while functional composition is modeled in the case of $\mathsf{CA}$s, by a family of operators $q_n$ of arity $n+1$ (for $n \geq 0$), and in $\aleph_0$-$\mathsf{AC}$s, by a single operator $q$ of arity $\omega$.

The axioms of $\mathsf{CA}$s and $\aleph_0$-$\mathsf{AC}$s characterize, up to isomorphism, algebras of functions, which are called functional $\mathsf{CA}$s  and functional $\aleph_0$-$\mathsf{AC}$s, respectively. The elements of a functional $\mathsf{CA}$ or $\aleph_0$-$\mathsf{AC}$ over a value domain $A$ are functions $\varphi: A^\omega \to A$ of arity $\omega$, referred to as $\omega$-ary functions. In this setting, the nullary operators $\e_i: A^\omega \to A$ are the projections, while the $q_n,q$  operators are given by finitary, infinitary composition of $\omega$-ary functions, respectively:
$$\begin{array}{l}
    q_n(\psi,\varphi_1,\ldots,\varphi_n)(s)=\psi(\varphi_1(s),\ldots,\varphi_n(s),s_{n+1},\ldots)\\
    q(\psi,\varphi_1,\ldots,\varphi_n,\varphi_{n+1},\ldots)(s)=\psi(\varphi_1(s),\ldots,\varphi_n(s),\varphi_{n+1}(s),\ldots)
  \end{array}
  $$
for every $s\in A^\omega$. 
The universe of a functional $\mathsf{CA}$  is called an $\omega$-clone, while the universe of a functional $\aleph_0$-$\mathsf{AC}$  is  called an infinitary $\omega$-clone. It is important to note that every infinitary $\omega$-clone is also an $\omega$-clone.
Furthermore, infinitary $\omega$-clones extend the concept of clone of functions: every clone can be encoded into an appropriately constructed infinitary $\omega$-clone.

\medskip

The starting point of this work is the  observation that, given an (infinitary) $\omega$-clone $C$ over $A$,   we can obtain a monoid structure on the set  of the
$\omega$-sequences of elements of $C$ as follows:
$$((\varphi_i)_{i\in\omega}\circ(\psi_j)_{j\in\omega})(s)=(\varphi_i(\psi_0(s),\ldots,\psi_n(s),\ldots))_{i\in\omega}$$
for every $s\in A^\omega$. This is just ordinary composition of functions from $A^\omega$ to $A^\omega$, up to  the identification of any function  $\varphi: A^\omega\to A^\omega$ with  an $\omega$-sequence $\varphi_0,\varphi_1,\ldots,\varphi_n,\ldots$  of functions from $A^\omega$ to $A$.

We introduce the variety of \emph{merge monoids} (abbreviated as m-monoids),  that encompasses abstract clones, clone algebras, and also monoids in the degenerate case. 
An m-monoid $\mathbf M=(M,\cdot, 1, \star_n,\bar\sigma)_{n\geq 0, \sigma\in S_\omega}$ is a one-sorted algebra of type $(2,0,2,1)$, where $(M,\cdot,1)$ forms a monoid,
 $S_\omega$ is the set of finite permutations of $\omega$,
and $(M,\star_n,\bar\sigma)_{n\geq 0, \sigma\in S_\omega}$ constitutes a so-called \emph{merge algebra},  capturing the following structural properties of $\omega$-sequences: $x\star_n y$ is meant to denote the sequence obtained by replacing the first $n$ elements of $y$ by the first $n$ elements of $x$, and $\bar\sigma(x)$ is the $\omega$-sequence obtained by the $\sigma$-permutation of the elements of $x$.
The connection between the monoid structure and the merge  operation $\star_n$ is given by the right distributivity law:
$$(x\star_n y)\cdot z= (x\cdot z)\star_n (y\cdot z).$$
An m-monoid is degenerate when  $x\star_n y=y$ and $\bar\sigma(x)=x$ for every $n$ and $\sigma\in S_\omega$. The category of degenerate m-monoids is equivalent to the category of monoids.

We then explore the interaction between permutations and multiplication in two different ways: 
\begin{itemize}
\item In \emph{clone merge monoids} (\emph{cm-monoids}) we require $\bar\sigma(x\cdot y)=\bar\sigma(x)\cdot y$, inspired by the paradigmatic example of all endofunctions of $A^\omega$, for some set $A$.
  In this case, as suggested above, the 
  multiplication of the monoid is  the composition of functions, 
while the merge algebra structure is obtained by the identification of any function  $\varphi: A^\omega\to A^\omega$ with  an $\omega$-sequence $\varphi_0,\varphi_1,\ldots,\varphi_n,\ldots$  of functions from $A^\omega$ to $A$.
Given an $\omega$-clone $F$, the set of $\omega$-sequences of elements of $F$ is the universe of a cm-monoid of this form.
Remarkably, we shall prove   that all non-degenerate cm-monoids are necessarily noncommutative. 
\item In \emph{arithmetical merge monoids} (\emph{am-monoids}) we instead require $\bar\sigma(x\cdot y)=\bar\sigma(x)\cdot \bar\sigma(y)$. This is motivated by the following  example from arithmetics:
the multiplicative monoid of positive natural numbers, where the merge structure arises from the prime factorisation of a number.
\end{itemize}

As explained in the rest of this introduction, the main results of the paper concern cm-monoids.
In summary, the category of cm-monoids fits naturally in the following picture:

\medspace

\begin{center}
\begin{tabular}{|c|ccc|}
\hline
&\it{Algebraic}&&\it{Categorical}\\
\hline
\it{Many-sorted}&Abstract clones& $\cong$ &Lawvere theories\\
&\raisebox{0.225cm}{\rotatebox{-90}{$\subset$}}&&\\
\it{One-sorted}&Clone algebras, $\aleph_0$-$\mathsf{AC}$s&$\subset$&cm-Monoids\\
\hline
\end{tabular}
\end{center}

\medspace

In this table, we use the term ``algebraic'' to describe structures that abstract 
 finitary  ($A^n\to A$) or infinitary ($A^\omega\to A$) functions. In contrast, 
 the categorical column abstracts structures based on finitary
   $(A^n\to A^m)$ or infinitary $(A^\omega\to A^\omega)$ functions. 
   Additionally, the inclusion symbols represent categorical embeddings, while the symbol $\cong$ signifies categorical equivalence. 

\medskip

We now outline the main contributions of this work, as well as the plan of the rest of the paper. 
After covering some preliminaries, we revisit the concepts of clone algebras and $\aleph_0$-$\mathsf{AC}$s in Sections \ref{sec:pre} and \ref{sec:clonealg}. There, we establish a categorical adjunction between the category of abstract clones and the category of all clone algebras, which restricts to an equivalence with finite-dimensional clone algebras. In Section \ref{sec:merge}, we  introduce the new notion of \emph{merge algebra}, along with essential concepts such as rank, coordinate and extensionality. Our central object of study emerges in Section \ref{sec:var}, where we study the main properties of  m-monoids, cm-monoids and am-monoids.
In Section \ref{sec:dim}  we introduce the concept of finite dimensionality for m-monoids, which enables us to recast abstract clones in the framework of cm-monoids.

The core results of this work are presented in Sections \ref{sec:ca1} and \ref{sec:ca2}.
In Section \ref{sec:ca1}, we establish a categorical adjunction between clone algebras and cm-monoids, which restricts to an equivalence with finitely ranked cm-monoids.  This equivalence narrows  to one between finite-dimensional clone algebras  and finite-dimensional, finitely ranked cm-monoids.
 By combining these results with the findings from Section \ref{sec:clonealg}, we deduce an equivalence between abstract clones and finite-dimensional, finitely ranked cm-monoids.
In Section \ref{sec:ca2}, we introduce \emph{partial infinitary clone algebras} ($\mathsf{PICA}$s), a unifying framework that generalises both clone algebras and $\aleph_0$-$\mathsf{AC}$s. We prove that the category of $\mathsf{PICA}$s is equivalent to the category of extensional cm-monoids. The table below summarizes these results.

\medskip

\begin{center}
\begin{tabular}{|ccc|}
  \hline
  &&\\
abstract clones& $\cong$ & finite-dimensional, finitely ranked cm-monoid\\
\raisebox{0.225cm}{\rotatebox{-90}{$\subset$}}&&\raisebox{0.225cm}{\rotatebox{-90}{$\subset$}}\\
  clone algebras&$\cong$&finitely ranked cm-monoids\\
  \raisebox{0.225cm}{\rotatebox{-90}{$\subset$}}&&\raisebox{0.225cm}{\rotatebox{-90}{$\subset$}}\\
 partial infinitary clone algebras&$\cong$&extensional cm-monoids\\
&&\\
\hline
\end{tabular}
\end{center}

\medskip

In future work, we will develop the theory of abstract polymorphisms and invariant relations through the notion of \emph{$\mathbf M$-module}, where $\mathbf M$ is a cm-monoid. An $\mathbf M$-module is a merge algebra $\mathbf B$ equipped with an action of $\mathbf M$ on $B$. 
A relevant example of an $\mathbf M$-module, where $\mathbf M$ is a suitable cm-monoid of endofunctions of $A^\omega$, is the following:
for a relation $R\subseteq A^\omega$, we consider the merge algebra $\mathbf B$  of all $\omega\times\omega$-matrices over $A$, whose columns belong to $R$.
A function from $A^\omega$ to $A^\omega$ acts on an $\omega\times\omega$ matrix by row-wise functional application. 
The largest cm-monoid $\mathbf M$ of endofunctions of $A^\omega$  such that $\mathbf B$ is an $\mathbf M$-module defines the set $M$ of all abstract polymorphisms of the relation $R$.
In this way we can recast the classical theory of polymorphisms and invariant relations.

\section{Preliminaries}\label{sec:pre}
The notation and terminology in this paper are pretty standard. For
concepts, notations and results not covered hereafter, the reader is
referred to \cite{BS81,mac87} for universal algebra, to \cite{L06,FMMT22a,SZ86,T93} for the theory of clones, and to \cite{BS22} for clone algebras. 

In this paper, the symbol $\omega$ represents the first infinite ordinal. A finite ordinal $n<\omega$ is understood as the set of elements $\{0, \ldots, n-1\}$.

In the rest of this section, the symbol $A$ will be used to denote an arbitrary set.

\subsection{Traces and $\omega$-sequences}\label{sec:tra}

\begin{enumerate}
\item Given an $\omega$-sequence $s\in A^\omega$ and elements $a_i\in A$, we define $s[a_0,\dots,a_{n-1}]\in A^\omega$ by
  $s[a_0,\dots,a_{n-1}]_i=a_i$ if $0\leq i\leq n-1$;  $s_i$ otherwise.
    
 \item For $a\in A$, we denote by $a^\omega\in A^\omega$ the constant sequence defined by $(a^\omega)_i = a$ for all $i$. 
 
\item If $X\subseteq A^\omega$, then we define $\mathrm{dom}(X)= \{s_i: s\in X, i\in\omega\}$.

\item The relation $\equiv$ on $A^\omega$, defined by
$s\equiv r\ \Leftrightarrow\ |\{i: s_i\neq r_i\}|< \omega$,
is an equivalence relation.  The equivalence class of an element $s\in A^\omega$  will be denoted by $[s]_\equiv$.

\item A \emph{trace on $A$} is a subset $X$  of $A^\omega$ closed under the relation $\equiv$:
$\ s\in X\ \text{and}\ r\equiv s \Rightarrow r\in X$.

\item  A set $X$ of $\omega$-sequences is called a \emph{trace} if $X$ is a trace on the set $\mathrm{dom}(X)$.

\end{enumerate}

\subsection{Permutations}

\begin{enumerate}
\item A permutation $\sigma$ of $\omega$ is \emph{finite} if $\mathrm{dom}(\sigma)=\{ n: \sigma(n)\neq n\}$ is a finite set. We often write $\sigma_n$  for $\sigma(n)$.
Note that $\mathrm{cod}(\sigma)=\sigma(\mathrm{dom}(\sigma))=\mathrm{dom}(\sigma)$. 
We denote by $\iota$  the identity permutation and by $\tau^n_k$    the transposition such that $\tau^n_k(k)=n$, $\tau^n_k(n)=k$ and  $\tau^n_k(i)=i$ for every $i\neq k,n$.

\item The set $S_\omega$ of all finite permutations of $\omega$ is generated by all transpositions. The transpositions $\tau^k_0$ ($k\in\omega$) also generate $S_\omega$.
 A permutation $\sigma$ of $\omega$ is a
permutation of $Y$ if $\mathrm{dom}(\sigma)\subseteq Y\subseteq \omega$.

\item Two permutations $\sigma,\rho\in S_\omega$ are disjoint if 
$\mathrm{dom}(\sigma)\cap \mathrm{dom}(\rho)=\emptyset$. If $\sigma,\rho\in S_\omega$ are disjoint, then $\sigma\circ\rho = \rho\circ\sigma$.

\item If $\sigma$ is a permutation, then $\sigma^A: A^\omega\to A^\omega$ is the map defined by
 $\sigma^A(s)=(s_{\sigma_i}: i\in \omega)$. 
\end{enumerate}

\subsection{Finitary functions, $\omega$-ary functions and clones} \label{sec:operations}
\begin{enumerate}
\item A \emph{finitary function on $A$} is a function $f:A^n\to A$ for some $n\in\omega$. The set of all finitary functions on $A$ is denoted by $O_A$.  Given $F\subseteq O_A$, we define $F^{(n)}=\{f\in F\mid f:A^n\to A \}$.

 \item An  \emph{$\omega$-ary function on $A$} is a function $\varphi:A^\omega\to A$. The set of all $\omega$-ary functions on $A$ is denoted by $O_A^{(\omega)}$. The identity endofunction of $A^\omega$  is denoted by $\mathrm{Id}$.

\item For a finitary function $f: A^n\to A$,  the \emph{top extension $f^\top$ of $f$} is the $\omega$-ary function defined by $f^\top(s)=f(s_0,\dots,s_{n-1})$ for every $s\in A^\omega$. For $F\subseteq O_A$,  we write $F^\top=\{f^\top : f\in F\}$.

\item The $\omega$-ary function $\e^A_n:A^\omega\to A$ is the projection defined by $\e^A_n(s)=s_n$ for every $s\in A^\omega$.

\item Given a function $f:A\to B$, the function $f^\omega:A^\omega\to B^\omega$ is defined as follows: $f^\omega(s)=(f(s_i): i\in\omega)$ for every $s\in A^\omega$.

\item   A \emph{clone of functions} on $A$ is a subset $C$ of $O_A$ containing all projections
$\e^{n}_i:A^n\to A$ ($0\leq i \leq n-1$) and closed under composition. 
 Given two clones $C$ and $D$, a \emph{clone homomorphism} is a mapping $F : C\to D$ preserving the arity of functions, mapping the
projections in $C$ to the corresponding projections in $D$ and preserving
composition.
\end{enumerate}

\subsection{Abstract clones}\label{sec:attempt} 

We recall from  \cite{T93} and \cite[p.\,239]{evans82} that an \emph{abstract clone}  is a many-sorted algebra composed of  disjoint sets $B_n$ (for $n\geq 0$),
distinguished elements $\e_i^{n}\in B_n$ (for $n\geq 1$ and $0\leq i\leq n-1$), and
a family of operations $q_n^k: B_n\times (B_k)^n \to B_k$ (for all $k$ and $n$)
such that
\begin{enumerate}
\item[(A1)] $q_n^k(q_m^n(x,y_0,\dots,y_{m-1}),\mathbf z)= q_m^k(x,q_n^k(y_0,\mathbf z),\dots,q_n^k(y_{m-1},\mathbf z))$, where $x$ is a variable of sort $m$, $y_0,\dots,y_{m-1}$ of sort $n$ and $\mathbf z$ is a sequence of variables of sort $k$;
\item[(A2)] $q^n_n(x,\e_0^{n},\dots,\e_{n-1}^{n})=x$, where $x$ is a variable of sort $n$;
\item[(A3)] $q_n^k(\e_i^{n},y_0,\dots,y_{n-1})=y_i$, where $y_0,\dots,y_{n-1}$ are variables of sort $k$.
\end{enumerate}
We call the equations in (A1), (A2), (A3) associativity, right unit and left unit laws, respectively.
A morphism of abstract clones is a morphism of many-sorted algebras, thus given by a collection of functions commuting with the operations
$\e_i^n$ and $q_n^k$.

\smallskip
The primary example of an abstract clone is a clone of functions,
with $\e_i^n$ as above (cf. Section \ref{sec:operations}) and
$$q_n^k(f,g_0,\dots,g_{n-1})(x_0,\ldots,x_{k-1})=f(g_0(x_0,\ldots,x_{k-1}),\ldots,g_{n-1}(x_0,\ldots,x_{k-1})).
$$

\subsection{Abstract $\aleph_0$-clones \cite{neu70,T93}}\label{sec:neu}
 An \emph{abstract $\aleph_0$-clone} ($\aleph_0$-$\mathsf{AC}$) is an infinitary algebra $\mathbf B =(B,q^\mathbf B,\e^\mathbf B_i)$, where the $\e^\mathbf B_i$ ($1\leq i<\omega$) are nullary operators and $q^\mathbf B$ is an infinitary operation satisfying:
 \begin{enumerate}
 \item[(N1)] $q(\e_i,x_1,\dots,x_n,\dots)=x_i$;
 \item[(N2)]  
 $q(x,\e_1,\dots,\e_n,\dots)=x$;
 \item[(N3)]
 $q(q(x,\mathbf  y),\mathbf z)= q(x,q(y_1,\mathbf z),\dots,q(y_n,\mathbf z),\dots)$,
 where  $\mathbf y$ and $\mathbf  z$ are countable infinite sequences of variables.  
 \end{enumerate}

 A \emph{functional} $\aleph_0$-$\mathsf{AC}$  with value domain $A$ is an algebra
 $\mathbf F=(F,q^\mathbf F,\e_i^\mathbf F)$, where $F\subseteq  O_A^{(\omega)}$, $\e^\mathbf F_i(s)=s_i$, and for every $\varphi,\psi_i\in F$ and every $s\in A^\omega$,
 $q^\mathbf F(\varphi,\psi_0,\dots,\psi_n,\dots)(s)=\varphi(\psi_0(s),\dots,\psi_n(s),\dots)$.

 Neumann shows in \cite{neu70} that every abstract $\aleph_0$-$\mathsf{AC}$ is isomorphic to a functional  $\aleph_0$-$\mathsf{AC}$  and that  there is a faithful functor from the category of clones to the category of abstract $\aleph_0$-$\mathsf{AC}$s, but this functor is not full.  The connection between abstract $\aleph_0$-$\mathsf{AC}$s and clone algebras is explained in \cite[Section 4.3]{BS22} (see also Section \ref{sec:ca2}).

\section{Clone algebras}\label{sec:clonealg}

In this section, we recall from \cite{BS22} the definition of a \emph{clone algebra} as a more canonical algebraic account of clones using standard one-sorted algebras, and we study the relationship between clone algebras and abstract clones. In particular, we revisit and refine the analysis of such relationships, adapting and extending previous approaches in light of our framework. Moreover, we show that every lambda abstraction algebra \cite{PS93} is a clone algebra.

The finitary type of  clone algebras contains a countably infinite family of nullary operators $\e_n$ ($n\geq 0$) and, for each $n\geq 0$, an operator $q_n$ of arity $n+1$.  

Throughout this paper, whenever we write
$q_n(x,\mathbf y)$, it is implicitly assumed that  $\mathbf y=y_0,\dots,y_{n-1}$ is a sequence of length $n$.

 \begin{definition} \label{def:clonealg} 
 A \emph{clone algebra}\footnote{We have made a small notational change with respect to the definition given in \cite{BS22}, where the 0-ary operations $\e_i$ are indexed starting from 1 rather than 0. The new notation matches conventions on infinite sequences ranging over  $\omega$ rather than $\omega\setminus\{0\}$.} ($\mathsf{CA}$, for short)  is an algebra 
 $\mathbf C = (C, q^\mathbf C_n,\e^\mathbf C_n)_{n\geq 0}$, where $C$ is a set, $\e^\mathbf C_n$ is an element of $C$, and $q^\mathbf C_n$ is an $(n+1)$-ary  operation satisfying the following identities: 
\begin{enumerate}
\item[(C1)] $q_n(\e_i,x_0,\dots,x_{n-1})=x_i$ $(0\leq i\leq n-1)$;
\item[(C2)] $q_n(\e_j,x_0,\dots,x_{n-1})=\e_j$ $(j\geq n)$;
\item[(C3)] $q_n(x,\e_0,\dots,\e_{n-1})=x$ $(n\geq 0)$;
\item[(C4)] $q_n(x,  \mathbf{y})= q_k(x, \mathbf{y},\e_{n},\dots,\e_{k-1})$ ($k> n$);
 \item[(C5)] $q_n(q_n(x, y_0,\dots,y_{n-1}), \mathbf{z})=q_n(x,q_n(y_0, \mathbf{z}),\dots,q_n(y_{n-1},\mathbf{z}))$.
\end{enumerate}
\end{definition}

For convenience, we will use the simplified notation $\mathbf C = (C, q^\mathbf C_n,\e^\mathbf C_n)$,  instead of explicitly writing out the indexing $\mathbf C = (C, q^\mathbf C_n,\e^\mathbf C_n)_{n\geq 0}$.

An $\omega$-clone on a set $A$ is a subset $C\subseteq O_A^{(\omega)}$ that satisfies the following properties: 
  \begin{itemize}
\item $C$ contains all projections $\e_n^A:A^\omega\to A$;
\item $C$ is closed under the family of operations $q_n^A$ of arity $n+1$, defined by:
  $$q^A_n(\varphi,\psi_0,\dots,\psi_{n-1})(s)=\varphi(s[\psi_0(s),\dots, \psi_{n-1}(s)]),$$
   for every $s\in A^\omega$ and $\varphi,\psi_0,\dots,\psi_{n-1}\in  C$.
  \end{itemize}
If $C$ is an $\omega$-clone, then the algebra $(C, q_n^A, \e^A_n)_{n\geq 0}$ is called a \emph{functional clone algebra $(\mathsf{FCA})$ with value domain $A$} (see \cite{BS22}).
One of the main results in \cite{BS22} states that every $\mathsf{CA}$
is isomorphic to some $\mathsf{FCA}$.

\begin{example}\label{exa:proj} The algebra $\mathbf P=(\omega, q_n^\mathbf P,\e_n^\mathbf P)$, where  $\e_n^\mathbf P=n$ and $q_n^\mathbf P(i,k_0,\dots,k_{n-1})=k_i$ if $i< n$, else $i$,
is the minimal  clone algebra.
\end{example}

\begin{example}\label{exa:free} Let $\rho$ be a finitary type of algebras, $K$ be a variety of $\rho$-algebras and  $\mathbf F_K=(F_K,\sigma^\mathbf F)_{\sigma\in\rho}$ be the free $K$-algebra over a countable set $I=\{v_0,v_1,\dots\}$ of generators. 
An endomorphism $f$ of $\mathbf{F}_K$ is called  $n$-finite if $f(v_i)=v_i$ for every $i\geq n$. These $n$-finite endomorphisms can be collectively expressed by an $(n+1)$-ary operation $q_n^\mathbf F$ on $F_K$ as follows (see   \cite[Definition 3.2]{mac83} and \cite[Definition 5.2]{BS22}): $q_n^\mathbf F(a,b_0,\dots,b_{n-1})=s(a)$ for every $a,b_0,\dots,b_{n-1}\in F_K$,
where $s$ is the unique  endomorphism
of $\mathbf{F}_K$ mapping the generator $ v_i$ to $b_i$ ($0\leq i\leq n-1$). 
 Then the algebra  $\mathrm{Cl}(\mathbf{F}_K)=(F_K, q_n^\mathbf{F}, \e_n^\mathbf{F})$ is a clone algebra,
where $\e_n^\mathbf{F}= v_n\in I$.
The operation $\sigma^\mathbf F$ (for $\sigma\in\rho_n$) is represented by the equivalence class in $\mathbf F_K$ of the  $\rho$-term $\sigma(v_0,\dots,v_{n-1})$.
\end{example}

\begin{example} \label{exa:topext} Every clone of functions $C$ on a set $A$ determines an $\mathsf{FCA}$, whose universe is the set $C^\top=\{f^\top: f\in C\}$ of top extensions of functions in $C$ (see item (3) in Section \ref{sec:operations}). Roughly speaking,  this corresponds to adding to each function $f \in C$ an infinite number of dummy arguments. The top extension operator commutes with the composition:
$$g(f_0,\dots,f_{n-1})^\top= g^\top(f_0^\top,\dots,f_{n-1}^\top,\e_n^A,\e^A_{n+1},\dots)$$
for every $n$-ary function $g\in C$ and $k$-ary functions $f_0,\dots,f_{n-1}\in C$.
\end{example}

\begin{example}\label{exa:laa}
 In this example, we show that every lambda abstraction algebra ($\mathsf{LAA}$) (see \cite{PS93,PS95}) is a clone algebra. This result is important within the line of research that investigates which varieties of algebras can be interpreted in the lambda calculus. For example, the variety of  groups cannot be interpreted in the lambda calculus,   whereas the varieties of  lattices and of subtractive algebras can  (see \cite[Section 6.2]{BS06} and \cite{Sel03}). Some of the most important open problems in lambda calculus concern the interpretability of suitable varieties of algebras, for instance the order incompleteness problem \cite{Sel03} is equivalent  to the interpretability of $n$-permutability, for some $n\geq 2$. 

Let $\mathbf A$ be an $\mathsf{LAA}$.
We have to distinguish the variables of lambda calculus from the algebraic ones.  The former are written $x,y,\ldots$ and the latter $X,Y,\ldots$. We fix a  linear order on the set of variables of lambda calculus: $x_1,x_2,\ldots,x_n,\ldots$. We recall that in $\mathsf{LAA}$s, we have a binary operator of application and countably many unary operators $\lambda x_i$, for each variable $x_i$ of lambda calculus. As usual, $t_1t_2t_3\ldots t_n$ stands for $(\ldots ((t_1 t_2)t_3)\ldots t_n)$ and $\lambda x_1x_2\dots x_n.t$ for $\lambda x_1(\lambda x_2(\dots(\lambda x_n.t)\dots))$.
For every $n\in\omega$, we define:
 $$q_n(X,Y_1,\dots,Y_n)=(\lambda x_1\dots x_n.X)Y_1\dots Y_n;\qquad \e_n= x_n.$$
By applying \cite[Theorems 13,14]{SA00}
one can verify quite easily that  $(A,q_n,\e_n)_{n\in\omega}$ is a clone algebra.

\end{example}

\subsection{Independence and dimension}
We define the notions of independence and dimension in clone algebras, abstracting the notion of arity of the finitary functions. We follow \cite[Section 3]{BS22}.

\begin{definition} \cite[Definition 3.4]{BS22} An element $a$ of a clone algebra $\mathbf  C$ \emph{is independent of} $\e_{n}$ if 
$q_{n+1}(a,\e_0,\dots,\e_{n-1},\e_{n+1})=a$. If $a$ is not independent of $\e_{n}$, then we say that 
$a$ \emph{is dependent on} $\e_{n}$.
\end{definition}

In \cite[Lemma 3.5]{BS22} it is shown that  $a$ is independent of $\e_n$ iff $q_{n+1}(a,\e_0,\dots,\e_{n-1},b)=a$ for every $b\in C$. We define the \emph{dimension of $a$} as follows:
$$\mathrm{dim}(a)=\begin{cases}0&\text{if $\{k: \text{$a$ depends on $\e_k$}\}=\emptyset$} \\
n+1&\text{if $n=\mathrm{max}\{k: \text{$a$ depends on $\e_k$}\}$} \\\omega&\text{if $\{k: \text{$a$ depends on $\e_k$}\}$ is infinite.} \end{cases}
$$

An element $a\in C$ is \emph{finite-dimensional} if $\mathrm{dim}(a)<\omega$.
We denote by $ C_{\mathrm{fd}}$ the set of all finite-dimensional elements of a clone algebra $\mathbf C$.
 The set $ C_{\mathrm{fd}}$ is a subalgebra of $\mathbf C$.
 We say that $\mathbf C$ is \emph{finite-dimensional} if $C= C_{\mathrm{fd}}$.

 The clone algebras introduced in Examples \ref{exa:proj}, \ref{exa:free} and \ref{exa:topext} are all finite-dimensional.

\subsection{Abstract clones and clone algebras} \label{sec:ac-ca}

We now construct  a categorical equivalence between abstract clones and finite-dimensional clone algebras -- a result first established in \cite[Theorem 3.20]{BS22}. There are three differences and improvements in the present treatment with respect to \cite{BS22}: (1) we show how to construct a clone algebra from an arbitrary abstract clone in such a way that our general construction generalises the  construction of a clone algebra from a clone of functions (see Remark \ref{rem:approx-top}); (2) thanks to this generalisation, rather than using Evans' representation theorem
(\cite[Theorem 2]{evans89}) stating that each abstract clone is isomorphic to one arising from a clone of functions, we obtain a new proof of this result as a direct corollary of our constructions (see Corollary \ref{cor:evans}); (3) the equivalence arises from a ``larger"  adjunction between the category of abstract clones and the category of \emph{all} clone algebras. Since this is a reformulation of a known result, the proofs in this subsection are only outlined.

\medskip
We first recall from \cite{BS22} the construction associating to a clone algebra $\mathbf{C}$ a clone of functions $R_{\mathbf{C}}$, which is a fortiori  an abstract clone. Let $a\in C$ be an element of dimension $\leq k$. We define a function $\tilde a^{(k)}:C^k\rightarrow C$ as follows:
$\tilde a^{(k)}(x_0,\ldots,x_{k-1})=q_k(a,x_0,\ldots,x_{k-1})$,
and we set $(R_{\mathbf C})^{(k)} = \{ \tilde a^{(k)}: a\in C\:\textrm{with dimension}\:\leq k\}$.
The indexed set $R_C$ of functions forms a  clone of functions. In particular, we have
$\tilde a^{(n)}(\tilde b_0^{(k)},\ldots,\tilde b_{n-1}^{(k)})=\tilde c^{(k)}$, where $c=q_n(a,b_0,\ldots,b_{n-1})$.
The following remark will be useful.
\begin{remark} \label{rem:tilde-injective}
Note that, if $\tilde a^{(k)}=\tilde b^{(k)}$, then $a=\tilde a^{(k)}(\e_0,\ldots,\e_{k-1})=\tilde b^{(k)}(\e_0,\ldots,\e_{k-1})=b$.
\end{remark}

\medskip
We next go in the converse direction, from abstract clones to clone algebras. Let $\mathbf{B}$ be an abstract clone. For every   $k\geq 0$ and $x\in B_n$, we define 
$x^{+k}= q_n^{n+k}(x,\e_0^{n+k},\ldots,\e_{n-1}^{n+k})\in B_{n+k}$ and $x^\star=\{x^{+k} : k\geq 0\}$. 
 It is routine to prove the following properties: 
(a)  $(x^{+k})^{+n}=x^{+(k+n)}$; (b) if $x^{+k}=y^{+n}$ ($k\geq n$), then 
$x^{+(k-n)}=y$.

Let   $B=\bigcup_{n\geq 0} B_n$ and $X_\mathbf{B}=\{x\in B\ |\ \forall y\in B \ \forall  n>0\!:\  x\neq y^{+n}\}$.  
When $\mathbf B$ is a clone of functions,  $X_\mathbf{B}$ is the set of all functions depending on their last argument.

\begin{remark} \label{e-+}
  We observe that $(\e_i^n)^{+k}=\e_i^{n+k}$ for all $0\leq i<n$ and $k$. Moreover, for all $n$, $\e^{n}_{n-1}\in X_\mathbf{B}$ and $\e^{n}_j\not\in X_\mathbf{B}$ for all $j<n-1$.
\end{remark}

\begin{remark}\label{RC}
 Let   $R_{\mathbf{C}}$ be the  clone of functions associated to a clone algebra $\mathbf{C}$. For every $a\in C$, we have  $\tilde a^{(k)}\in X_{R_{\mathbf{C}}}$ iff $a$ has dimension $k$.
\end{remark}

It is easy to show the following lemma.

\begin{lemma} 
The family of sets $\{x^\star\}_{x\in X_\mathbf{B}}$ is a partition of $B$.
\end{lemma}

Given $y\in B$, we denote by $y^-$ the element of $X_\mathbf{B}$ such that
$y\in {(y^-)}^\star$.

\medskip
We define $\mathbf{B}^{\textrm{ac-ca}}=(B^{\textrm{ac-ca}},q_n^{\mathbf{B}^{\textrm{ac-ca}}},\e_n^{\mathbf{B}^{\textrm{ac-ca}}})_{n\geq 0}$ as follows:
\begin{itemize}
\item $B^{\textrm{ac-ca}}=X_\mathbf{B}$;
\item $\e_n^{\mathbf{B}^{\textrm{ac-ca}}}=\e_n^{n+1}$ (for all $n\geq 0$);
\item  $q_n^{\mathbf{B}^{\textrm{ac-ca}}}(x,x_0,\ldots,x_{n-1})$ is defined as follows. Given $y\in x^\star,y_0\in x_0^\star,\ldots,y_{n-1}\in x_{n-1}^\star$ chosen all to belong  to $B_k$ for some $k\geq n$, we set 
$q_n(x,x_0,\ldots,x_{n-1})=q_k^k(y,y_0,\ldots y_{n-1},\e_n^k,\ldots \e_{k-1}^k)^{-}$.
\end{itemize}

\begin{proposition}  \label{prop:arity-dimension}
 $\mathbf{B}^{\textrm{ac-ca}}$ is well-defined and is a finite-dimensional clone algebra. Moreover,  $x\in X_\mathbf{B}\cap B_k$  if and only if  $x$ has dimension $k$ in $\mathbf{B}^{\textrm{ac-ca}}$.
\end{proposition}
\begin{proof}
  We check that $q_n^{\mathbf{B}^{\textrm{ac-ca}}}$  is well-defined. Let
  $x,x_0,\ldots,x_{n-1}\in X_{\mathbf B}$, and $y\in x^\star,y_0\in x_0^\star,\ldots,y_{n-1}\in x_{n-1}^\star$ be  elements of  $B_k$ for some $k\geq n$. We have to prove that, for $n\leq k<k'=k+l$,  the elements 
$A  =  q_{k'}^{k'}(y^{+l}, y_0^{+l},\ldots,  y_{n-1}^{+l},\e_n^{k'},\ldots \e_{k-1}^{k'})$ and
$C = (q_k^k(y,y_0,\ldots y_{n-1},\e_n^k,\ldots \e_{k-1}^k))^{+l}$ are equal.
Unfolding the definitions of $y^{+l}=q_k^{k'}(y,\e_0^{k'},\ldots,\e_{k-1}^{k'})$ and $C=q_k^{k'}(q_k^k(y,y_0,\ldots y_{n-1},\e_n^k,\ldots \e_{k-1}^k),\e_0^{k'},\ldots,\e_{k-1}^{k'})$, and applying axioms (A1) and  (A3), we obtain indeed that both elements $A$ and $C$  are equal to $q_k^{k'}(y,y_0^{+l},\ldots,y_{n-1}^{+l},\e_n^{k'},\ldots,\e_{k-1}^{k'})$.
The straightforward verification of the  axioms (C1) to (C5) is left to the reader.

We show now that $\mathbf{B}^{\textrm{ac-ca}}$ is finite-dimensional. Let $x\in X_\mathbf{B}\cap B_k$, $k'> k$ and $y\in X_\mathbf{B}\cap B_m$. We  have, by definition of $q_n$ and using (A1)  and Remark \ref{e-+} (for big enough $l$):
$$\begin{array}{lll}
 q_{k'}(x,\e_0,\ldots,\e_{k'-2},y) & =& q_l^l(x^{+(l-k)},(\e^1_0)^{+(l-1)},\ldots, (\e_{k'-2}^{k'-1})^{+(l-k'+1)},y^{+(l-m)},\e_{k'}^l,\ldots,\e_{l-1}^l)^-\\
& =&  q_l^l(q_k^l(x,\e_0^l,\ldots,\e_{k-1}^l),\e_0^l,\ldots, \e_{k'-2}^l, y^{+(l-m)},\e_{k'}^l,\ldots,\e_{l-1}^l)^- \\
& =_{(A3)}& q_k^l(x,\e_0^l,\ldots, \e_{k-1}^l)^-\\
& =_{\textit{def}} & (x^{+(l-k)})^-\\
& =& x.
\end{array}$$
Thus $x$  has dimension $\leq k$. We now show that $x$ has indeed dimension $k$.
Assume by contraposition that $q_k(x,\e_0,\ldots,\e_{k-2},y)=x$ for every $y$. If we define $z=  q_k^{k-1}(x,\e^{k-1}_0,\dots,\e^{k-1}_{k-2},\e^{k-1}_{k-2})$, then $x= z^{+1}$ and hence $x\not\in X_\mathbf{B}$. Putting all together, we have proved that $\mathbf{B}^{\textrm{ac-ca}}$ is finite-dimensional, and that $x\in X_\mathbf{B}\cap B_k$  if and only if $x$ has dimension $k$. 
\end{proof}

\begin{remark}\label{rem:approx-top}
One checks readily that our construction $(\_)^{\textrm{ac-ca}}$ generalises the construction of a clone algebra
$F^\top$ from a clone of functions $F$ given in \cite[Section 3]{BS22} (see also Example \ref{exa:topext}). Viewing $F$ as an abstract clone, we have $F^\top=F^{\textrm{ac-ca}}$, 
because $f^-= g^-$ in our sense if and only if $f^\top=g^\top$. 
\end{remark}

\begin{lemma} \label{lem:preparation-ac-ca-adjunction}
\begin{enumerate}
\item  If $\mathbf{B}$ is an abstract clone, the map $\eta_{\mathbf{B}}:\mathbf{B}\rightarrow R_{\mathbf{B}^{\textrm{ac-ca}}}$, defined by
$\eta_{\mathbf{B}}(x) = \widetilde{x^-}^{(k)}$ for every $x\in B_k$, is an isomorphism of abstract clones.
\item If $\mathbf{C}$ is a clone algebra, the map $f_{\mathbf{C}}: \mathbf{C}_{\mathrm{fd}}\rightarrow (R_{\mathbf{C}})^{\textrm{ac-ca}}$, defined  on an element $a$ of dimension $k$ by
$f_{\mathbf{C}}(a)=\tilde a^{(k)}$, is an isomorphism of clone algebras.
\end{enumerate}
\end{lemma}

\begin{proof} 
The proof that $\eta_{\mathbf{B}}$ and $f_{\mathbf{C}}$ are morphisms in the respective categories is left to the reader. Suppose that 
$\eta_{\mathbf{B}}(x) = \eta_{\mathbf{B}}(y)$, for $x,y\in B_k$. By Remark \ref{rem:tilde-injective}, we have
$x^-=y^-$, and then $x=y$. We now show the  surjectivity of $\eta_{\mathbf{B}}$. By definition, any element of $(R_{\mathbf{B}^{\textrm{ac-ca}}})_k$ is of the form $\widetilde{y}^{(k)}$ for some $y\in X_\mathbf{B}$. Since $y=y^-$, we get the conclusion. 

We now move to the second claimed isomorphism. By Remark \ref{RC} the map $f_{\mathbf{C}}$ is well-defined. 
Surjectivity is immediate by definition. If $\tilde a^{(k)}=
\tilde b^{(k)}$, then $a=b$ by Remark \ref{rem:tilde-injective}.
\end{proof}

\begin{corollary} \label{cor:evans}
Every abstract clone is isomorphic to some clone of functions.
\end{corollary}
\begin{proof} Immediate by  Lemma \ref{lem:preparation-ac-ca-adjunction}(1).
\end{proof}

Let $\mathbb{ACL}$ and  $\mathbb{CA}$ be the categories of abstract clones and of clone algebras, respectively, whose morphisms are the morphisms preserving the multi-sorted and one-sorted structure, respectively.
We denote by $\mathbb{CA}^{\mathrm{fd}}$ the full subcategory of $\mathbb{CA}$ whose objects are the 
finite-dimensional clone algebras.

The constructions $R_{(\_)}$ (in which each $R_{\mathbf{C}}$ is viewed as an abstract clone) and 
$(-)^{\textrm{ac-ca}}$ can be extended to functors, as follows:
\begin{itemize}
\item Given a morphism $h:\mathbf{C}_1\rightarrow\mathbf{C}_2$ of clone algebras, we define the morphism of abstract clones $R_h:R_{\mathbf{C}_1}\rightarrow R_{\mathbf{C}_2}$ as follows:  $R_h(\tilde a^{(k)})= \widetilde{h(a)}^{(k)}$.  This is well-defined, because the dimension of $h(a)$ is less than or equal to the dimension of $a$.
\item
Given a morphism $g:\mathbf{B}_1\rightarrow \mathbf{B}_2$ of abstract clones, we define a morphism of clone algebras
$g^{\textrm{ac-ca}}: (\mathbf{B}_1)^{\textrm{ac-ca}} \rightarrow (\mathbf{B}_2)^{\textrm{ac-ca}}$ as follows:
$g^{\textrm{ac-ca}}(x)=g(x)^-$ for any $x\in X_{\mathbf{B}_1}$. 
\end{itemize}

\begin{theorem} \label{prop:ca-ac-adjunction}
The functor $(\_)^{\textrm{ac-ca}}$ is left adjoint to the functor $R_{(\_)}:\mathbb{CA}\rightarrow\mathbb{ACL}$.
The adjunction restricts to an equivalence between $\mathbb{ACL}$ and $\mathbb{CA}^{\mathrm{fd}}$.
\end{theorem}

\begin{proof}
In reference to Lemma \ref{lem:preparation-ac-ca-adjunction}, we take $\eta$ as unit, and we define the counit $\epsilon$ by $\epsilon_{\mathbf{C}}=(f_{\mathbf{C}})^{-1}$, considered as a map from $(R_{\mathbf{C}})^{\textrm{ac-ca}}$ to $\mathbf{C}$. Thus $\epsilon_{\mathbf{C}}$ is injective but not surjective in general. However,
we note that if $\mathbf{C}$ is finite-dimensional, then $\epsilon_{\mathbf{C}}$ is an iso, from which the final part of the statement follows immediately.
We check the triangular identities  
for $\eta$ and $\epsilon$.

(1) $R_{\epsilon_{\mathbf{C}}} \circ \eta_{R_{\mathbf C}}=id_{R_{\mathbf C}}$.
Unfolding the definitions, we get that $\eta_{R_{\mathbf C}}=R_{f_{\mathbf{C}}}$, hence
$R_{\epsilon_{\mathbf{C}}} \circ \eta_{R_{\mathbf C}}=R_{(\epsilon_{\mathbf{C}}\circ f_{\mathbf{C}})}$. By definition of $\epsilon$, $\epsilon_{\mathbf{C}}\circ f_{\mathbf{C}}:\mathbf{C}_{\mathrm{fd}}\rightarrow \mathbf{C}$ is an inclusion, which is mapped to the identity by $R_{(\_)}$ because by definition we have
$R_{\mathbf{C}_{\mathrm{fd}}}=R_{\mathbf{C}}$.

(2)  $\epsilon_{\mathbf{B}^{\textrm{ac-ca}}}\circ
(\eta_{\mathbf{B}})^{\textrm{ac-ca}}= id_{\mathbf{B}^{\textrm{ac-ca}}}$. The argument is similar and simpler. We note that
$(\mathbf{B}^{\textrm{ac-ca}})_{\mathrm{fd}}=\mathbf{B}^{\textrm{ac-ca}}$, 
so that $\epsilon_{\mathbf{B}^{\textrm{ac-ca}}}$ is the inverse of $f_{\mathbf{B}^{\textrm{ac-ca}}}$. We also have, again unfolding the definitions, that
$(\eta_{\mathbf{B}})^{\textrm{ac-ca}}=f_{\mathbf{B}^{\textrm{ac-ca}}}:\mathbf{B}^{\textrm{ac-ca}}\rightarrow (R_{\mathbf{B}^{{\textrm{ac-ca}}}})^{\textrm{ac-ca}}$. 
\end{proof}

\section{The variety of merge algebras}\label{sec:merge}
We  introduce \emph{merge algebras} as an algebraic 
abstraction of sets of $\omega$-sequences.

\begin{definition} A \emph{merge algebra} is a tuple $\mathbf A=(A,\star_n^{\mathbf A}, \bar\sigma^{\mathbf A})_{n\in\omega,\sigma\in S_\omega}$, where   $\bar \sigma^{\mathbf A}$ is a unary operation for every finite permutation $\sigma\in  S_\omega$ and $\star_n^{\mathbf A}$ is a binary operation  for every $n\geq 0$,
satisfying the following conditions:
\begin{itemize}
\item[(B1)] 
 $(A,\star_n^{\mathbf A})$ is an idempotent semigroup.
\item[(B2)] $x\star_0 y=y$.
\item[(B3)] ($k\geq n$): $(x\star_k y)\star_n z= x \star_n z$ and $x\star_k (y\star_n z)= x\star_k z$.
\item[(B4)] ($k< n$): $(x\star_k y)\star_n z= x\star_k (y \star_n z)$.
\item[(B5)] $\bar\sigma(\bar\tau (x))=\overline{\tau\circ \sigma}(x)$ and $\bar \iota (x)=x$, where $\iota$ is the identity permutation.
\item[(B6)]  Let $n\leq k$, $\sigma$ be a permutation of $k$, and $f:k\to \{x,y\}$ be the map such that $f_i=x$ iff $\sigma(i)< n$.  Then:
$$\bar\sigma(x\star_n y)= ((\dots(\bar\sigma(f_0) \star_1\bar\sigma(f_1) )\star_2\dots)\star_{k-1} \bar\sigma( f_{k-1}))\star_k y.$$
\item[(B7)] If $\sigma(i)=\tau(i)$ for every $m\leq i< n$, then: 
$$(z\star_m\bar\sigma( x))\star_n y=(z\star_m \bar\tau(x))\star_n y.$$

\end{itemize}
\end{definition}

From (B1) and (B3) it follows that each $(A,\star_n^{\mathbf A})$ is a rectangular band, i.e., an idempotent semigroup satisfying the identity:  $x\star_n y\star_n z= x \star_n z$.

Axioms (B4) and (B3) together imply that for any sequence 
$m_1\leq m_2\leq\dots\leq m_k$  of natural numbers,  we can write unambiguously:
\begin{equation}\label{ninetto}x_1\star_{m_1} x_2 \star_{m_2} \dots x_{k}\star_{m_k} x_{k+1},\end{equation}
as any possible parenthesing of this sequence yields the same result.

The map $\sigma \mapsto \bar\sigma$ is an antimorphism from the group $S_\omega$ 
of finite permutations to the group of bijective endomaps of $A$, meaning that it reverses composition: $\overline{\tau\circ \sigma}= \bar\sigma\circ\bar\tau$. 

\subsection{The canonical merge algebras}
In this section, we introduce the intended model of  identities (B1)-(B7).
 Let $X$ be a set, and $X^\omega$ denote  the set of $\omega$-sequences on $X$. We define the algebra  $\mathbf{Seq}(X)=(X^\omega, \star_n^{\mathrm{seq}}, \bar\sigma^{\mathrm{seq}})$ as follows, for every $s,u\in X^\omega$:
\begin{itemize}
\item $s \star_n^{\mathrm{seq}} u=(s_0,\dots,s_{n-1},u_n,u_{n+1},\dots)$;
\item $\bar\sigma^{\mathrm{seq}}(s)=(s_{\sigma_i}:i\in\omega)$.
\end{itemize}
We omit the superscript $^{\mathrm{seq}}$ when no ambiguity arises.
The algebra $\mathbf{Seq}(X)$ is indeed a merge algebra,
 called \emph{the fully canonical merge algebra on $X$}.

\begin{definition}
  A merge algebra $\mathbf B$ is called
\emph{canonical} if  it  is a subalgebra of a fully canonical merge algebra $\mathbf{Seq}(X)$ for some set $X$. 
\end{definition}

In the next lemma we show that the universe of a canonical merge algebra is a trace (see Section \ref{sec:pre}).

\begin{lemma}\label{lem:tra}
  Let $\mathbf{Seq}(X)$ be the fully canonical merge algebra on $X$, and let $Y\subseteq X^\omega$. Then $Y$ is a merge subalgebra of $\mathbf{Seq}(X)$ iff $Y$ is a trace on a subset of $X$.
\end{lemma}

\begin{proof}
 Let $Y$ be a trace on $Z\subseteq X$. If $s\in Y$ and $\sigma\in S_\omega$ is a permutation of $n$, then $\bar\sigma(s)= s[s_{\sigma_0},\dots, s_{\sigma_{n-1}}]\in Y$, because $s_{\sigma_i}\in Z$. Moreover, if $s,u\in Y$, then $s\star_n u= u[s_0,\dots,s_{n-1}]\in Y$, because $s_i\in Z$.
 
 Let $Y$ be a merge subalgebra  of $\mathbf{Seq}(X)$. We show that $Y$ is a trace on $\mathrm{dom}(Y)$. Let $s\in Y$ and $a_0,\dots,a_{n-1}\in \mathrm{dom}(Y)$. For each $a_i$ there exists $u^i\in Y$ and $h(i)\in \omega$ such that $u^i_{h(i)}=a_i$.
 Let $\gamma_i$ be the transposition $(i,h(i))$.
Then the sequence $\overline{\gamma_i}(u^i)$ has $a_i$ at position $i$.
We define by induction on $0\leq i\leq n-1$:
$v_0= \overline{\gamma_0}(u^0);\qquad v_{i}= v_{i-1}\star_{i} \overline{\gamma_{i}}(u^i)$. By construction $v_0,\ldots v_{n-1} \in Y$. Since $v_{n-1}\star_n s=s[a_0,\dots,a_{n-1}]$,  then $Y$ is a trace.
\end{proof}

\subsection{Some consequences of the axioms}

The following lemmas will be useful in the sequel.
\begin{lemma}\label{lem:exall}
Every merge algebra satisfies the following properties: 
\begin{enumerate}
\item $(\exists z.\ x\star_n z=y\star_n z) \Rightarrow (\forall z.\ x\star_n z=y\star_n z)$.
\item The operations $\star_k$ and $\star_n$ commute: $(x\star_n y)\star_k(x'\star_ny')=(x\star_k x')\star_n(y\star_k y')$.
\end{enumerate}
\end{lemma}
\begin{proof} (1)  Assume $x\star_n c=y\star_n c$ for some $c$. Then, using (B3), we have: $x\star_n z= x\star_n c\star_n z= y\star_n c\star_n z=  y\star_n z$. Thus, $x\star_n z=y\star_n z$
for all $z$.

(2) Let $k> n$. Then we have: 
\[
\begin{array}{lll}
 (x\star_n y)\star_k(x'\star_ny') & =_{(B3)}  &  (x\star_n y)\star_k y' \\
  &  =_{(B3)} & ((x\star_k x')\star_ny)\star_k y'  \\
  & =_{(B4)}  &   (x\star_k x')\star_n(y\star_k y').
\end{array}
\]
A similar argument applies when $k\leq n$.
\end{proof}

\begin{lemma}\label{lem:b8b12} Let $\sigma$ and $\tau$ be finite permutations. Every merge algebra satisfies the following identities:
\begin{enumerate}
\item If $\sigma$ is a permutation of $n$, then $\bar\sigma(x\star_n y)=\bar\sigma(x)\star_n y$.

\item If  $\sigma$  is a permutation of $\omega\setminus n$,  then  $\bar\sigma(x\star_n y)=x\star_n\bar\sigma(y)$.

\item If $\sigma(i)=\tau(i)$ for every $i\geq n$, then $x\star_n \bar\sigma(y)=x\star_n \bar\tau(y)$.

\item If $\sigma$ is a permutation of $n$ or of $\omega\setminus n$, then $\bar\sigma(x\star_n y)=\bar\sigma(x)\star_n \bar\sigma(y)$.

\item   If $0\leq j\leq k\leq m$ and $\sigma$ is a permutation of $m$ such that $\sigma(i)\geq k$ for every $0\leq i< j$, then
 $\bar\sigma (x \star_k y) \star_j z=\bar\sigma(y)\star_j z$.

\item $x\star_{n+1} y = x\star_n \bar\tau^n_0(\bar\tau^n_0(x)\star_1 z)\star_{n+1}y$.

\end{enumerate}
\end{lemma}

\begin{proof}

(1) $\bar\sigma(x\star_n y)=_{(B6)} \bar\sigma(x)\star_1\bar\sigma(x)\star_2\dots\star_{n-1}\bar\sigma(x)\star_n y=_{(B1)}\bar\sigma(x)\star_n y$.

\smallskip\noindent
(2) Let $k\geq n$, and let $\sigma$ be a permutation of  $k$ such that $\sigma(i)=i$ for every $0\leq i\leq n-1$. Then:\\
$\bar\sigma(x\star_n y)=_{(B6)} \bar\sigma(x)\star_1\dots\star_{n-1}\bar\sigma(x)\star_{n}\bar\sigma(y) \dots\star_{k-1}\bar\sigma(y)\star_k y=\bar\sigma(x)\star_{n}\bar\sigma(y) \star_k y=_{(B7)} \bar\iota(x)\star_{n}\bar\sigma(y) \star_k y=x\star_{n}\bar\sigma(y) \star_k y=x\star_{n}(\bar\sigma(y) \star_k y)=_{(1)} x\star_{n}\bar\sigma(y\star_k y)=x\star_{n}\bar\sigma(y) $.

\smallskip\noindent
(3) Let $\sigma,\tau$ be permutations of $k\geq n$. Then:
$\bar\sigma(y)= \bar\sigma(y\star_k y)=_{(1)} \bar\sigma(y)\star_k y$. Similarly, for $\tau$.  Hence, $x\star_n \bar\sigma(y)= x\star_n \bar\sigma(y)\star_k y=_{(B7)}x\star_n \bar\tau(y)\star_k y= x\star_n \bar\tau(y)$.

\smallskip\noindent
(4) Let $\sigma$ be a permutation of $n$. Then: $\bar\sigma(x\star_n y)=_{(1)}\bar\sigma(x)\star_n y=_{(B5)}\bar\sigma(x)\star_n \bar\iota(y)=_{(3)}\bar\sigma(x)\star_n \bar\sigma(y)$. Similarly, the result holds if $\sigma$ is a permutation of $\omega\setminus n$.

\smallskip\noindent
 (5)
We apply (B6) to $\bar\sigma (x \star_k y)$. Define $f:m\to \{x,y\}$ such that $f_i=x$ iff $\sigma(i)< k$.
Then $f_i=y$ for every
$0\leq i< j$. Consequently, 
$\bar\sigma (x \star_k y) \star_j z=(\bar\sigma(y)\star_1\dots\star_{j-1}
\bar\sigma(y)\star_j\bar\sigma(f_j)\dots\star_{m-1}\bar\sigma(f_{m-1})\star_m y)\star_j z=_{(B3)} \bar\sigma(y)\star_1\dots\star_{j-1}
\bar\sigma(y)\star_j z=_{(B1)}\bar\sigma(y)\star_j z$.

\smallskip\noindent
(6) If $n=0$, the identity holds trivially. Assume $n\geq 1$.
We apply (B6) to $\bar\tau^n_0(\bar\tau^n_0(x)\star_1 z)$. 
Define $f:m\to \{\bar\tau^n_0(x),z\}$ such that $f_i=\bar\tau^n_0(x)$ iff $\bar\tau^n_0(i)=0$. Then
  $f_i= z$ for every $0\leq i\leq n-1$ and $f_n=\bar\tau^n_0(x)$. Thus, $\bar\tau^n_0(\bar\tau^n_0(x)\star_1 z)=\bar\tau^n_0(z)\star_1\bar\tau^n_0(z)\star_2\dots\star_{n-1}\bar\tau^n_0(z)\star_n \bar\tau^n_0(\bar\tau^n_0(x))\star_{n+1}z=\bar\tau^n_0(z)\star_1\bar\tau^n_0(z)\star_2\dots\star_{n-1}\bar\tau^n_0(z)\star_n x\star_{n+1}z =_{(B1)}
 \bar\tau^n_0(z)\star_n x\star_{n+1}z $. 
  From this we get the conclusion:
  $x\star_n \bar\tau^n_0(\bar\tau^n_0(x)\star_1 z)\star_{n+1}y= x\star_n  \bar\tau^n_0(z)\star_n x\star_{n+1}z \star_{n+1} y =x \star_{n+1}y$.
  \end{proof}

\subsection{Restrictions}\label{sec:res}

This section introduces restrictions in the context of merge algebras by fixing a distinguished element called the "coordinator" and studying how it interacts with the algebraic structure. 

\begin{definition}
 A \emph{pointed merge algebra} is a pair $(\mathbf A,1^\mathbf A)$, where $\mathbf A$ is a merge algebra and $1^\mathbf A$ is a fixed element of $A$, referred to as the \emph{coordinator}.
\end{definition}

Given the coordinator $1$, we write  $x_{\pref n}$ for $x\star_n 1$.
The element $x_{\pref n}$ is called \emph{the $n$-restriction of $x$}.

\begin{lemma}\label{lem:emme} The following properties hold in every pointed merge algebra $(\mathbf A,1^\mathbf A)$:
\begin{itemize}
\item[(1)] $(x_{\pref n})_{\pref n}=x_{\pref n}$.
\item[(2)] $(x_{\pref n})_{\pref n+1}= x_{\pref n}$; hence,  $(x_{\pref n})_{\pref n+k}= x_{\pref n}$ for every $k\geq 0$.
\item[(3)] $(x_{\pref n+1})_{\pref n}=x_{\pref n}$; hence,  $(x_{\pref n+k})_{\pref n}= x_{\pref n}$ for every $k\geq 0$.
\item[(4)]  $1\star_n x= x$ (resp. $1\star_n x=1$) iff $x_{\pref n}=1$ (resp.  $x_{\pref n}= x$).
\item[(5)]  If $\sigma(0)\geq k$, then $\bar\sigma( x_{\pref k})_{\pref 1}= \bar\sigma(1)_{\pref 1}$.

\end{itemize}
\end{lemma}

\begin{proof} 
(1) $(x_{\pref n})_{\pref n}= x\star_{n} 1\star_{n} 1
=x\star_{n} 1= x_{\pref n}$.

\smallskip\noindent
(2) $(x_{\pref n})_{\pref n+1}= x\star_{n} 1\star_{n+1} 1
= x\star_{n} 1= x_{\pref n}$.

\smallskip\noindent
 (3) $(x_{\pref n+1})_{\pref n}= (x\star_{n+1} 1)\star_n 1 =_{(B3)} x \star_n 1=  x_{\pref n}$.

  \smallskip\noindent
  (4) Let  $1\star_n x= x$. Then we have:
 $x_{\pref n}=x \star_n 1=_{hp} 1\star_n x \star_n 1=  1\star_n 1 =1$.
 Conversely, if  $x_{\pref n}= 1$, then $1\star_n x=_{hp}x_{\pref n} \star_n x = x\star_n 1\star_n x=x\star_n x= x$.
 A similar proof works for the other equivalence.
  
 \smallskip\noindent
 (5) Apply Lemma \ref{lem:b8b12}(5) to $j=1$ and $y=z=1$.
    \end{proof}

\begin{definition}
  An element $a\in A$ \emph{has rank $n$} if $n$ is the smallest natural number such that $a_{\pref n}=a$.
\end{definition}

The coordinator $1^\mathbf A$ is the unique element of rank $0$ by axiom (B2).
We denote by $A_{|n}$ the set of all elements of rank $\leq n$, and by $A_{|\omega}$ the set $\bigcup_{n\geq 0} A_{|n}$. It is easy to see that $A_{|n}$ is closed under $\star_k$  for all $k\in\omega$, and that
 $A_{|\omega}$ is a merge subalgebra of $\mathbf A$.

\subsection{Coordinates}\label{sec:coord}

In this section we define coordinates and present key lemmas that highlight their structural properties. Moreover, we discuss specific examples, such as the canonical merge algebra, to illustrate the role of coordinates.

\begin{definition}
  Let $(\mathbf A,1^\mathbf A)$ be a pointed merge algebra. 
For every $a\in A$, the element
$$a[n]=\bar\tau^n_0(a)_{\pref 1}$$
is called the \emph{$n$-coordinate of $a$}.
\end{definition}

The $n$-coordinate $a[n]$ of $a$ depends on the chosen coordinator $1^\mathbf A$, but
(B3) ensures that if two elements $a$ and $b$ have the same coordinates relative to one coordinator, they have the same coordinates relative to any other.
Moreover, every coordinate $a[n]$ has rank $\leq 1$, so that $a[n]\in A_{|1}$. Additionally, note that $a[0]=a_{\pref 1}$ for every $a$.

We define the sequence of coordinates $a_{\sq}=(a[0],\dots,a[n],\dots)$ for $a\in A$,
and  $A_{\sq}=\{a_{\sq} : a\in A\}$.

The following lemma provides fundamental properties of coordinates in a pointed merge algebra.

\begin{lemma}\label{lem:[k]} Let $(\mathbf A,1^\mathbf A)$ be a pointed merge algebra and $a,b\in A$.
\begin{itemize}
\item[(i)] $\bar\sigma ( a)[k]=a[\sigma_k]$. 

\item[(ii)] $(a \star_k b)[n]=\begin{cases}a[n]&\text{if $n<k$}\\ b[n]&\text{if $n\geq k$} \end{cases}$. In particular, $(a_{\pref k})[n]=\begin{cases}a[n]&\text{if $n<k$}\\
1[n]&\text{if $n\geq k$.}\end{cases}$
\item[(iii)] $a[n][0]=a[n]$ and $a[n][k] =1[k]$ for $k\neq 0$.
\end{itemize}

\end{lemma}

\begin{proof}
  
  (i) $\bar\sigma ( a)[k]=\bar\tau^k_0( \bar\sigma( a))_{\pref 1}=_{(B5)}\overline{\sigma \circ \tau^k_0}( a)_{\pref 1} =_{(B7,\textrm{with}\;m=0,n=1)} \bar\tau^{\sigma_k}_0( a)_{\pref 1} =a[\sigma_k]$.

  \smallskip\noindent
  (ii) Let $n < k$. Then $(a \star_k b)[n] =\bar\tau^n_0 (a \star_k b)_{\pref 1}=_{L.\ref{lem:b8b12}(1)} (\bar\tau^n_0( a)\star_k b)_{\pref 1}=_{(B3)} \tau^n_0( a)_{\pref 1} =a[n]$. Let $n \geq k\geq j=1$. Then $(a \star_k b)[n] =\bar\tau^n_0 (a \star_k b)_{\pref 1} =_{L.\ref{lem:b8b12}(5)} \bar\tau^n_0 ( b)_{\pref 1} = b[n]$.

  \smallskip\noindent    
    (iii) That $a[n][0]=a[n]$ follows from Lemma \ref{lem:emme}(1).
    If $k>0$, then $\tau^k_0(0)=k\geq 1$, so that we can apply Lemma \ref{lem:emme}(5) as follows:  
    $a[n][k] = (\bar\tau^k_0[(\bar\tau^n_0(a)_{\pref 1}])_{\pref 1}=
     \bar\tau^k_0(1)_{\pref 1}=1[k]$.
     \end{proof}

\begin{example} \label{coordinates-canonical}
 Let $(\mathbf{Seq}(X),1^{\mathrm{seq}})$ be the pointed fully  canonical  merge algebra on set $X$, where $1^{\mathrm{seq}}=(\epsilon_0,\dots,\epsilon_n,\dots)$ acts as the coordinator. The $n$-coordinate of $a\in X^\omega$ is: $a[n] = (a_n, \epsilon_1,\dots,\epsilon_n,\dots)$.
 Here, the $n$-coordinate picks out the 
$n$-th element of the sequence 
$a$ while defaulting to the coordinator for all other positions.
\end{example}

The following  results  provide deeper insights into how coordinates interact with restrictions, merges and permutations.

\begin{lemma}\label{lem:B12} Let $(\mathbf A,1^\mathbf A)$ be a pointed merge algebra.
\begin{enumerate}
\item  $x_{\pref n+1}= (x\star_n \bar\tau^n_0(x[n]))_{\pref n+1}$.

\item  $x\star_{n} y= x[0]\star_1\bar\tau^1_0(x[1])\star_2\dots\star_{n-1}\bar\tau^{n-1}_0(x[n-1])\star_{n} y$.

\item If $x[i] = y[i]$ for every $0\leq i\leq n-1$, then $\forall z\: (x\star_n z=y\star_n z)$.\\
In particular, If $x,y\in A_{|n}$ and $x[i] = y[i]$ for every $0\leq i\leq n-1$, then $x=y$.

\item    If $0<n\leq k$, then $$\bar\tau^k_0 (\bar \tau^k_0(x)\star_n x)=\bar\tau^k_0(x)\star_n x.$$

\item  If $\sigma$ is a permutation of $n+k$ and $\sigma(n+i)=i$ for every $0\leq i\leq k-1$, then 
$$x\star_n \bar\sigma(y) \star_{n+k}z= x\star_n \bar\sigma(y\star_k z).$$
\end{enumerate}
\end{lemma}

\begin{proof}
 (1) It follows from Lemma \ref{lem:b8b12}(6) by putting $y=z=1$.
 
   \smallskip\noindent
 (2) 
 By Lemma \ref{lem:exall} it is sufficient to prove the case $y=1$. 
 The proof is by induction on $n$. The case $n=0$ is trivial.
\[
\begin{array}{lll}
x\star_{n+1} 1&  =_{(1)} & x \star_n\bar\tau^n_0(x[n])\star_{n+1} 1\\
&  =_{(B1)} &(x\star_n 1)\star_n\bar\tau^n_0(x[n])\star_{n+1} 1\\
&  =_{Ind} & (x[0]\star_1\bar\tau^1_0(x[1])\star_2\dots\star_{n-1}\bar\tau^{n-1}_0(x[n-1])\star_n 1)\star_n\bar\tau^n_0(x[n])\star_{n+1} 1  \\
&  =_{(\ref{ninetto})} & x[0]\star_1\bar\tau^1_0(x[1])\star_2\dots\star_{n-1}\bar\tau^{n-1}_0(x[n-1])\star_n 1\star_n\bar\tau^n_0(x[n])\star_{n+1} 1  \\
&   =_{(B3)}   & x[0]\star_1\bar\tau^1_0(x[1])\star_2\dots\star_n\bar\tau^n_0(x[n])\star_{n+1} 1.  
\end{array}
\]

\noindent
(3) By (2).

  \smallskip\noindent
(4) We first note that
$$(\bar\tau^k_0(x)\star_n x)[i] = \left\{
\begin{array}{ll}
x[k] &\text{if $i=0$ or $i=k$}\\
x[i] &\text{if $i\neq 0,k$}
\end{array} \right\}= (\bar\tau^k_0 (\bar \tau^k_0(x)\star_n x))[i].$$
Then we get
$$\begin{array}{lll}
\bar\tau^k_0(x)\star_n x& =_{(B1)} & \bar\tau^k_0(x)\star_n x\star_{k+1} x\\
& =_{\textrm{(3)}} &  \bar\tau^k_0 (\bar \tau^k_0(x)\star_n x) \star_{k+1} x\\
& =_{L.\ref{lem:b8b12}(1)}& \bar\tau^k_0 (\bar \tau^k_0(x)\star_n x \star_{k+1} x)\\
& =_{(B1)}&  \bar\tau^k_0 (\bar \tau^k_0(x)\star_n x).
\end{array}
$$

  \smallskip\noindent
(5) The proof is very similar to the one of (4). One first easily checks that $x\star_n \bar\sigma(y) \star_{n+k} z$ and
$x\star_n \bar\sigma(y\star_k z)$ have the same coordinates. Then we get
$$\begin{array}{lll}
x\star_n \bar\sigma(y) \star_{n+k} z& =_{(B1)} & x\star_n \bar\sigma(y) \star_{n+k} z\star_{n+k} z\\
 &=_{\textrm{(3)}} &  x\star_n \bar\sigma(y\star_k z)\star_{n+k} z \\
 &=_{L.\ref{lem:b8b12}(1)}& x\star_n \bar\sigma(y\star_k z\star_{n+k} z)\\
 &=_{(B1)}&  x\star_n \bar\sigma(y\star_k z).
\end{array}
$$
\end{proof}

\begin{lemma}\label{lem:1ksigma}
 Let $(\mathbf A,1^\mathbf A)$ be a pointed merge algebra such that $1[k]=1$ for every $k$. Then $\bar\sigma(1)=1$ for every  $\sigma\in S_\omega$.
\end{lemma}

\begin{proof} Suppose that $1[k]=1$ for every $k$. Since
 by Lemma \ref{lem:B12}(4) we have $\bar\tau^k_0 (1[k])=1[k]$,  we get $\bar\tau^k_0 (1)=1$ for every $k$.  The conclusion follows, because the permutations $\tau^k_0$ generate $S_\omega$.
\end{proof}

\begin{lemma}\label{lem:4.16}
 Let $(\mathbf A,1^\mathbf A)$ be a pointed  merge algebra. Then the function $(-)_\sq: A\to \mathrm{Seq}(A_{|1})$, defined by $a_\sq = (a[0],a[1],\dots)$, is a homomorphism from $\mathbf A$ to the full canonical merge algebra $\mathbf{Seq}(A_{|1})$.
\end{lemma}

\begin{proof}
 By Lemma \ref{lem:[k]}.
\end{proof}

From Lemmas \ref{lem:[k]} and \ref{lem:tra} it follows that $A_\sq$ is a trace on $A_{|1}$.

\subsection{Degenerate merge algebras}

In this section, we introduce the concepts of degenerate and faithful merge algebras.  

\begin{definition} A merge algebra  is
\begin{itemize}
\item[(i)]   \emph{degenerate} if $\bar\sigma(x)=x$ for every $\sigma \in S_\omega$ and $x\star_n y=y$ for every $n\in\omega$;
\item[(ii)] \emph{faithful} if $\bar\sigma\neq \bar\tau$ for every $\sigma\neq\tau\in S_\omega$.
\end{itemize}
\end{definition}

If $|X|\geq 2$, the fully canonical merge algebra $\mathbf{Seq}(X)$ is faithful. 
Note that a degenerate merge algebra is not faithful. Hence, in particular,
a degenerate merge algebra cannot be fully canonical.

\begin{lemma}\label{lem:deg} Let $\mathbf A$ be a merge algebra. The following conditions are equivalent:
\begin{enumerate}
\item $\mathbf A$ is degenerate.
\item $\mathbf A\models \bar\sigma(x)=x$, for every $\sigma \in S_\omega$.
\item $\mathbf A\models x \star_n y=y$, for every $n\in\omega$.
\end{enumerate}
If $\mathbf A$ is pointed with coordinator $1^\mathbf A$, then the above conditions are also equivalent to the following:
\begin{enumerate}
\item[(4)] $\mathbf A\models x[k]=1$, for every $k\in\omega$.
\item[(5)] $\mathbf A\models x\star_1 1=1$.
\end{enumerate}

\end{lemma}

\begin{proof} 
(1 $\Rightarrow$ 3) Trivial.

\smallskip\noindent
(3 $\Rightarrow$ 2) Let $\sigma$ be a finite permutation. Then there exists $n$ such that $\sigma$ is a permutation of $n$.
By Lemma \ref{lem:b8b12}(1) and the hypothesis we get the conclusion: $\bar\sigma(y)=_{hp}\bar\sigma(x\star_n y)=_{L.\ref{lem:b8b12}(1)}\bar\sigma(x)\star_n y= y$.

\smallskip\noindent
(2 $\Rightarrow$ 1) We prove that $x\star_j z=z$, for every $j$. Let $\sigma$ be any permutation satisfying the following condition: there exists $k\geq j$ such that  $\sigma_i\geq k$ for every
  $0\leq i<j$. Then 
  $y\star_j z=_{hp}\bar\sigma(y)\star_j z =_{L.\ref{lem:b8b12}(5)}\bar\sigma (x \star_k y) \star_j z=_{hp}(x \star_k y) \star_j z =_{(B3)} x \star_j z$, because $k\geq j$. By putting $y=z$, we obtain the conclusion $x \star_j z=z\star_j z=z$.

This concludes the proof that (1)-(3) are  equivalent. We now assume that $\mathbf A$ is pointed and proceed to show that (4) and (5) are also equivalent to (1)-(3).

\smallskip\noindent
(1 $\Rightarrow$ 5) Trivial.

\smallskip\noindent
(5 $\Rightarrow$ 4): By $x[k]=\bar\tau^k_0(x)\star_1 1 =_{(5)} 1$ for every $x$ and $k$.

\smallskip\noindent
(4 $\Rightarrow$ 3) All elements of $A$ have the same coordinates. By Lemma \ref{lem:B12}(3) we get 
\begin{equation}\label{eq1} \forall xyz.\ x\star_n z=y\star_n z. \end{equation}
Therefore, by putting $y,z=1$ in (\ref{eq1}) we obtain $x\star_n 1 = 1$ for every $x\in A$. Moreover, by putting $y=1$ in (\ref{eq1}) we obtain $x\star_n z = 1\star_n z =_{L.\ref{lem:emme}(4)}z$ for arbitrary $x$ and $z$, because $z\star_n 1 = 1$. 
\end{proof}

It is no surprise that the category of sets and the category of degenerate merge algebras are isomorphic.

\subsection{The abstract trace}
We abstract out the notion of trace as defined in Section \ref{sec:tra}.

Let $\mathbf A$ be a merge algebra. We define a binary relation $\sim_\omega$ on $A$ as follows:
$$a\sim_\omega b\ \text{if}\ \exists n.\ b\star_n a = b.$$

\begin{example} \label{abstract-concrete-trace}
 For a canonical merge algebra, we have $\sim_\omega\ =\ \equiv$, that is $a\sim_\omega b$ iff $|\{i : a_i\neq b_i\}|<\omega$.
\end{example}

\begin{lemma} \label{lem:sim}
\begin{enumerate}
\item The relation $\sim_\omega$ is an equivalence relation. 
\item If $c\star_n a = b$ for some $c$ and $n$, then $a\sim_\omega b$.
	\item $a\sim_\omega (b\star_k a)$ for every $k\in\omega$ and $a,b\in A$.
	\item $a\sim_\omega b \Rightarrow \bar\sigma(a)\sim_\omega \bar\sigma(b)$.
\end{enumerate}
\end{lemma}

\begin{proof}  (1)
 (Reflexivity): $a\star_n a =_{(B1)} a$.

\smallskip\noindent
(Symmetry): If  $b\star_n a = b$,  then $a\star_n b=a\star_n b\star_n a=
a\star_n a=
a$. Therefore, $b\sim_\omega a$.
 
 \smallskip\noindent
(Transitivity) Let $b\star_n a = b$ and $c\star_k b = c$. Then, we have two cases:
\begin{itemize}

\item ($k< n$):  $c\star_n a = c\star_k b\star_n a=
c\star_k (b\star_n a)= c\star_k b=c$.

\item ($k\geq n$): $c =c\star_k b = c\star_k (b\star_n a)=_{(B3)} c \star_k a$.
\end{itemize}

\noindent
(2) $b\star_n a = c\star_n a\star_n a =
c\star_n a = b$.

\smallskip\noindent
(3) $ b\star_k a\star_k a=
b\star_k a$.

\smallskip\noindent
(4) Let $b\star_n a = b$. Then $b\star_{n+k} a=b$ for all $k\geq 0$. Indeed, we have
$b\star_{n+k} a= b\star_n a\star_{n+k} a=
b\star_n a = b$. Let $\sigma$ be a permutation of $m$ and $j>n,m$. Therefore
$\bar \sigma(b)\star_j \bar\sigma(a)=_{L.\ref{lem:b8b12}(4)} \bar\sigma(b\star_j a)=
\bar\sigma(b)$.
\end{proof}

\begin{definition}
 A subset of $A$ is called an \emph{abstract trace} of $\mathbf A$ if it is a union of $\sim_\omega$-equivalence classes.
\end{definition}

\begin{lemma}\label{lem:at}
(1) $A_{|\omega} = [1]_{\sim_\omega}$.

(2) Every abstract trace $Y$ of a merge algebra $\mathbf A$ is a merge subalgebra of $\mathbf A$ such that
  $A\star_n Y\subseteq Y$ for every $n$.
\end{lemma}

\begin{proof} 
(1) We have  $1\sim_\omega x$ iff $\exists n.\ x_{\pref n}= x$.

(2) Let $a\in A$ an $b\in Y$. By Lemma \ref{lem:sim}(3), we have
$a\star_n b \sim_\omega b$, so that $a\star_n b\in Y$. Moreover, if $\sigma$ is a permutation of $n$, then
$\bar\sigma(a) \star_n a= \bar\sigma(a\star_n a)= \bar\sigma(a)$; hence, $a\sim_\omega \bar\sigma(a)$.
\end{proof}

\section{Varieties  of m-monoids}\label{sec:var}

In this section we introduce {\em m-monoids}, {\em cm-monoids} and  {\em am-monoids}.

\subsection{m-Monoids}

We endow merge algebras with a monoid structure.

\begin{definition} 
A \emph{merge-monoid (m-monoid, for short)}  is a pair $\mathbf M=(\mathbf M_0,\mathbf M_1)$, where $\mathbf M_0=(M,\cdot^{\mathbf M},1^{\mathbf M})$ is a monoid  and $\mathbf M_1=(M,\bar\sigma^{\mathbf M}, \star_n^{\mathbf M})$ is a merge algebra, satisfying the following identity for every $n\in \omega$:
\begin{itemize}
\item[(RD)] Right Distributivity:  $(x\star_n y)\cdot z= (x\cdot z) \star_n( y \cdot z)$.
\end{itemize}
\end{definition}

The algebra $\mathbf M_0$  is called the \emph{multiplicative reduct of $\mathbf M$}, while
 $\mathbf M_1$ is called the \emph{merge reduct of $\mathbf M$}.

The merge reduct $\mathbf M_1$ of an m-monoid $\mathbf M$ will be implicitly considered as a pointed merge algebra in which the coordinator is the identity  $1^{\mathbf M}$ of the monoid.

\begin{example}\label{exa:funct}  (Functions from $X^\omega$ to $X^\omega$) 
  Let $X$ be a set and $F_X^{(\omega)}$ be the set  of all endofunctions of $X^\omega$. The set $F_X^{(\omega)}$ forms the universe of an m-monoid $\mathbf F_X^{(\omega)}$, where $(\mathbf F_X^{(\omega)})_0$ is the monoid of all endofunctions of $X^\omega$ and $(\mathbf F_X^{(\omega)})_1$ is isomorphic to the fully canonical merge algebra on the set $\{\psi:X^\omega\to X\ |\ \psi\in X^{X^\omega} \}$.
Indeed, a function $\varphi:X^{\omega}\to X^{\omega}$ is  identified with the $\omega$-sequence $(\varphi_0,\ldots,\varphi_n,\ldots)$, where $\varphi_i=\e_i^X\circ\varphi $.  
  \end{example}

We say that an m-monoid $\mathbf M$ is \emph{commutative} if its multiplicative reduct $\mathbf M_0$ is commutative, and that it is \emph{degenerate, canonical, fully canonical or faithful}  if its merge reduct  $\mathbf M_1$ has the corresponding property. Additionally, we say  that $x\in M$ has rank $n$ if it has rank $n$ in the merge reduct of $\mathbf M$.

Just as a degenerate merge algebra is simply a set, a degenerate m-monoid is simply a monoid. In fact, the categories of monoids and degenerate m-monoids are isomorphic.

We now prove some simple properties of m-monoids.

\begin{lemma}\label{lem:emme2} Let $\mathbf M$ be an m-monoid.
The following properties hold:
\begin{itemize}
\item[(1)] $( x \cdot y)_{\pref n}= (x_{\pref n} \cdot y)_{\pref n}$.
\item[(2)] $( x \cdot y_{\pref n})_{\pref n}= x_{\pref n} \cdot y_{\pref n}$.
\item[(3)] $x\sim_\omega y \Rightarrow (x\cdot z)\sim_\omega (y\cdot z)$, for every $z$. 
\end{itemize}
\end{lemma}

\begin{proof} (1) $(x_{\pref n} \cdot y)_{\pref n} = ((x\star_n 1)\cdot y)\star_n 1 =_{(RD)} ((x\cdot y) \star_n (1\cdot y))\star_n 1=(x \cdot y)\star_n y\star_n 1=(x \cdot y)\star_n 1=( x \cdot y)_{\pref n}$.

\smallskip\noindent
 (2)  $x_{\pref n} \cdot y_{\pref n} =(x\star_n 1)\cdot (y\star_n 1)=_{(RD)} (x \cdot (y\star_n 1))\star_n (y\star_n 1) =_{(B1)} (x \cdot (y\star_n 1))\star_n  1= (x \cdot y_{\pref n})_{\pref n}$.

  \smallskip\noindent
  (3) Let $x\sim_\omega y$. Then, by definition, there exists some $n$ such that $y\star_n x = y$. We have 
$(y\cdot z)\star_n (x\cdot z)= (y\star_n x)\cdot z=y\cdot z$. 
Thus, $(x\cdot z)\sim_\omega (y\cdot z)$.
 \end{proof}

\begin{definition} \label{def:lf} 
An m-monoid $\mathbf M$ is said to be \emph{finitely ranked} if $M=M_{|\omega}$.
\end{definition}

It is not difficult to verify that $M_{|\omega}$ is an m-submonoid of $\mathbf M$.

\subsection*{Extensional  m-monoids}

We now define extensional m-monoids, providing examples of both extensional and non-extensional m-monoids. Additionally, we establish a  connection between extensional m-monoids and canonical m-monoids, proving that an m-monoid is extensional if and only if it is (up to isomorphism) canonical. 

\begin{definition} An  m-monoid $\mathbf M$ is called \emph{extensional}  if the following implication holds for every $ x,y\in M$:
$$(\forall n\in\omega.\ x[n]=y[n]) \Rightarrow x=y.$$
\end{definition}

\smallskip
\begin{example} \label{exa:dsum}
(1) Any  nontrivial degenerate m-monoid is not  extensional.

(2) The m-monoid $\mathbf F_X^{(\omega)}$ from Example \ref{exa:funct} is extensional. 

(3)   We provide an example of a non-degenerate, non-extensional m-monoid.
  Let $\mathbf M$ be an m-monoid, and $M\oplus M=\{(x,i):x\in M, i=0,1\}$ be 
  the carrier of the algebra $\mathbf M\oplus \mathbf M$ in the type of m-monoids, where:
  $$(x,i)\star^\oplus_n (y,j) =(x\star_n y, j);\quad  \overline{\sigma}^\oplus(x,i)=(\overline{\sigma}(x),i);\quad (x,i)\cdot^\oplus(y,j)=(x\cdot y,(i+j)\text{ mod } 2);\quad 1^\oplus=(1,0).$$
  It is easy to see that 
  $\mathbf M\oplus \mathbf M$ is an m-monoid, which is never extensional, since for all $x\in M$ the elements $(x,0)$ and $(x,1)$ have the same coordinates.
\end{example}

\begin{lemma}\label{lem:coto} 
\begin{enumerate}
\item   An m-monoid is extensional iff  it is  (up to isomorphism) canonical.
\item Every finitely ranked m-monoid is extensional.
\end{enumerate}

\end{lemma}

\begin{proof} (1) Let $\mathbf M$ be an m-monoid, where $1^\mathbf M=(1_0,\ldots,1_n,\ldots)$.

\noindent  ($\Leftarrow$) 
For all $x\in M$, by canonicity we have $x[n] = (x_n,1_1,\dots,1_n,\dots)$.
  If  $x[n] = y[n]$ for all $n$, it follows that $x_n=y_n$ for all $n$, and hence $x=y$. Thus,  $\mathbf M$ is extensional.
  
 \smallskip\noindent
($\Rightarrow$) We define a canonical m-monoid $\mathbf M_{\sq}$ with universe $M_{\sq}=\{x_\sq: x\in M\}$, where $x_\sq=(x[i]: i\in\omega)$ (cf. Section \ref{sec:coord}).
By Lemma \ref{lem:4.16}, $M_{\sq}$ forms the universe of a canonical merge algebra on $M_{|1}$. Next, we introduce a binary operation on $M_{\sq}$ defined as
$x_{\sq}\cdot^{\sq} y_{\sq} = (x\cdot^\mathbf M y)_{\sq}=((x\cdot^\mathbf M y)[i]\ |\ i\in\omega)$.
This operation is well defined because $\mathbf M$ is extensional.  
Using Lemma \ref{lem:4.16}, the definition of $\cdot^\sq$ and the extensionality of $\mathbf M$, it follows that the map $x\mapsto x_{\sq}$ is an isomorphism of m-monoids from $\mathbf M$ to the canonical m-monoid $\mathbf M_{\sq}$.

(2) By Lemma \ref{lem:B12}(3).
\end{proof}

\subsection{The variety of cm-monoids}
We now  introduce the variety of cm-monoids, which is a variety of finitary algebras, generalising  monoids, clone algebras and  free algebras.

\begin{definition}
 A \emph{cm-monoid}  is an m-monoid satisfying the following identity   for every $\sigma \in S_\omega$: 
\begin{itemize}
\item[(LS)] $ \bar\sigma(x\cdot y) =\bar\sigma(x)\cdot y$.
\end{itemize}
\end{definition}

\begin{remark} \label{rem-L2-simplified}
We observe that (LS) is equivalent to 
$ \bar\sigma(x)= \bar\sigma(1)\cdot x$.
\end{remark}

\begin{remark} Note that the m-monoid $\mathbf{F}_X^{(\omega)}$ of Example \ref{exa:funct} is a cm-monoid.
  \end{remark}

\begin{lemma}\label{lem:[k]bis} Let  $\mathbf M$ be a cm-monoid. Then we have:
 \begin{enumerate}
 \item $(x\cdot y)[n] =(x[n]\cdot y)_{\pref 1}$.
 \item $x[n]=(1[n]\cdot x)_{\pref 1}$. 
\item $x[i] = y[j] \Rightarrow (x\cdot z)[i] = (y\cdot z)[j]$.
\item  The set $\bar S_\omega=\{\bar\sigma(1): \sigma\in S_\omega\}$  is a subgroup of the monoid $\mathbf M_0$.
\end{enumerate}
\end{lemma}

\begin{proof}
 (1) $(x \cdot y )[n] =\bar\tau^n_0(x \cdot y)_{\pref 1}=_{(LS)} (\bar\tau^n_0( x) \cdot y)_{\pref 1}=_{L.\ref{lem:emme2}(1)} (\bar\tau^n_0( x)_{\pref 1} \cdot y)_{\pref 1}= (x[n] \cdot y)_{\pref 1} $.
 
 \smallskip\noindent
(2)-(3) By (1).
 
 \smallskip\noindent
 (4) $\bar\sigma(1)\cdot \bar\tau(1) =_{(LS)}\bar\sigma(1\cdot\bar\tau(1)) = \bar\sigma(\bar\tau(1))=\overline{\tau\circ\sigma}(1)$. Also, taking suitable instantiations for $\sigma$ and $\tau$, we get that
 $\overline{\sigma^{-1}}(1)$ is the inverse of  $\bar\sigma(1)$.
  \end{proof}

\begin{lemma}\label{lem:faith1}
Let $\mathbf M$ be a non-degenerate cm-monoid. Then the following conditions hold:
\begin{enumerate}
\item $1[k]\neq 1[j]$ for every $k\neq j$.
\item $\mathbf M$ is faithful.
\item $\mathbf M$ is noncommutative.
\end{enumerate}
\end{lemma}

\begin{proof}
 (1) Assume that there exist $k\neq j$ such that $1[k]=1[j]$. Let $x$ be an arbitrary element of $M$.
Since $x[k]=(1[k]\cdot x)_{\pref 1}$ by Lemma \ref{lem:[k]bis}(2), we have $x[k]=x[j]$. 
 We now show that all the coordinates of $x$ are equal: if $i\neq k,j$, then $x[j]=x[\tau^i_k(j)]=_{L.\ref{lem:[k]}(i)}(\bar\tau^i_k(x))[j]= (\bar\tau^i_k(x))[k]=_{L.\ref{lem:[k]}(i)}x[\tau^i_k(k)]=x[i]$.
 It follows  that $1[k]=1$ for every $k$, because $1[0]=1$. 
 From Lemma \ref{lem:deg} it follows that $\mathbf M$ is degenerate and we get a contradiction.
 
  \smallskip\noindent
 (2)  If $\sigma\neq \tau$, then there exists $i$ such that $\sigma_i\neq\tau_i$. In such a case, we have $\bar\sigma(1)[i] =_{L.\ref{lem:[k]}(i)} 1[\sigma_i]\neq_{(1)} 1[\tau_i] =_{L.\ref{lem:[k]}(i)}\bar\tau(1)[i]$. Therefore, $\bar\sigma(1)\neq \bar\tau(1)$.
 
   \smallskip\noindent
 (3)  Let $\sigma \neq \tau$ be finite permutations such that $\sigma\circ\tau\neq \tau\circ\sigma$. Suppose that 
$\mathbf M$ is commutative. 
Then we have: $\overline{\tau\circ\sigma}(1)=_{(B5)}\bar\sigma( \bar\tau(1))=_{(LS)} \bar\sigma(1)\cdot \bar\tau(1)=_{\mathrm{comm}} \bar\tau(1)\cdot \bar\sigma(1)=\bar\tau(\bar\sigma(1))= \overline{\sigma\circ\tau}(1)$. It follows that  $\overline{\tau\circ\sigma}=\overline{\sigma\circ\tau}$, because $\overline{\tau\circ\sigma}(x)=\overline{\tau\circ\sigma}(1)\cdot x=\overline{\sigma\circ\tau}(1)\cdot x=\overline{\sigma\circ\tau}(x)$. 
Therefore, $\mathbf M$ is not faithful, contradicting (2).
\end{proof}

\begin{proposition}\label{prop:loc=str}
 Let $\mathbf M$ be an extensional m-monoid. Then $\mathbf M$  is a cm-monoid if and only if 
 $$(*)\quad \forall x,y,z\in M\: \forall i,j\in \omega\:(x[i]=y[j] \Rightarrow (x\cdot z)[i]=(y\cdot z)[j]).$$
\end{proposition}

\begin{proof} By Lemma \ref{lem:coto} we know that $\mathbf M$ is (up to isomorphism) canonical since $\mathbf M$ is assumed to be extensional.

\smallskip\noindent
  ($\Leftarrow$) Let $x\in M$ and $1=(1_0,1_1,\dots,1_n,\dots)$. We have to show that  $\bar\sigma(x) = \bar\sigma(1)\cdot x$ (cf. Remark \ref{rem-L2-simplified}).
 Since $1_{\sigma(k)}=\bar\sigma(1)_k$, we have: 
 $\bar\sigma(x)_k=x_{\sigma_k} = (1\cdot x)_{\sigma_k} =_{(*)}(\bar\sigma(1)\cdot x)_k$ for every $k$. Therefore, we get the conclusion
since $\mathbf M$ is  canonical. 
 
 \smallskip\noindent
 ($\Rightarrow$) Let $x,y\in M$ such that $x_i=y_j$. We prove that $(x\cdot z)_i=(y\cdot z)_j$ for every $z$.
 Since $\mathbf M$ is canonical,  we have
  $(1\star_i x)\star_{i+1} 1  = (1\star_i \bar\tau^i_j(y))\star_{i+1} 1$. If we multiply by $z$ on the right and apply (twice) the right distributivity, we get $(z\star_i (x\cdot z))\star_{i+1} z = (z\star_i (\bar\tau^i_j(y)\cdot z))\star_{i+1} z$.
 By again applying canonicity we have $(x\cdot z)_i = (\bar\tau^i_j(y)\cdot z)_i$. We obtain the conclusion if we prove that
 $(\bar\tau^i_j(y)\cdot z)_i = (y\cdot z)_j$. Indeed, since $\mathbf M$ satisfies (LS),
 we have 
$(\bar\tau^i_j(y)\cdot z)_i = (\bar\tau^i_j(y\cdot z))_i = (y\cdot z)_j$.
This completes the proof.
\end{proof}

\begin{example} 
 Let $\mathbf A$ be a merge algebra. Then the pair constituted by the monoid  $(A^A, \circ, \mathrm{Id})$ of all endofunctions of $A$ and the merge algebra $\mathbf A^A$, product of $A$-copies of $\mathbf A$, is  a cm-monoid denoted by 
 $End(\mathbf A)$. The cm-monoid $End(\mathbf A)$  will be called the \emph{endofunction  cm-monoid of the merge algebra $\mathbf A$}.
 The endofunction cm-monoid $End(\mathbf{Seq}(X))$ of a full canonical merge algebra $\mathbf{Seq}(X)$ is isomorphic to the cm-monoid $\mathbf F_X^{(\omega)}$ of all endomaps of $X^\omega$ (see Example \ref{exa:funct}).
\end{example}

\subsection{The variety of am-monoids}
In this section we introduce another class of m-monoids, which we call \emph{arithmetical m-monoids} (am-monoids, for short). The introduction of this class of m-monoids is justified by Example \ref{exa:ar} below. 

\begin{definition}
An \emph{am-monoid}  is an m-monoid satisfying the following identity   for every $\sigma \in S_\omega$: 
\begin{itemize}
\item[(A1)] $\bar\sigma(x\cdot y) =\bar\sigma(x)\cdot \bar\sigma(y)$.
\end{itemize}
An am-monoid  is \emph{strong} if it satisfies the following further identity:
\begin{itemize}
\item[(A2)] $(x\cdot y)\star_n 1= (x\star_n 1)\cdot (y\star_n 1)$.
\end{itemize}
\end{definition}

\begin{lemma}\label{lem:am} Let $\mathbf M$ be an am-monoid. Then the following properties hold:
 \begin{itemize}
 \item[(i)] $\bar\sigma(1)=1$ for every permutation $\sigma$.
  \item[(ii)] The set $\{\bar\sigma: \sigma\in S_\omega\}$ is a group of automorphisms of the multiplicative reduct $\mathbf M_0$.
 \item[(iii)] $(x\cdot y)[n] =(\bar\tau^n_0(x)\cdot \bar\tau^n_0(y))_{\pref 1}=(x[n]\cdot \bar\tau^n_0(y))_{\pref 1}$.
  \item[(iv)] If $\mathbf M$ is strong, then $(x\cdot y)[n] =x[n]\cdot y[n]$.
\end{itemize}
\end{lemma}

\begin{proof}
(i) For every $x\in M$, we have $\bar\sigma(x)= \bar\sigma(1\cdot x)=\bar\sigma(1)\cdot \bar\sigma(x)$ and $\bar\sigma(x)= \bar\sigma(x\cdot 1)=\bar\sigma(x)\cdot \bar\sigma(1)$. Since $\bar\sigma$ is bijective,
we derive that $y\cdot\bar\sigma(1)=\bar\sigma(1)\cdot y=y$ for every $y\in M$.
Since the unit is unique, we obtain $\bar\sigma(1)=1$ for every $\sigma$.

\smallskip\noindent
(ii) By (i) and (A1).

\smallskip\noindent
 (iii)
  $(x\cdot y)[n] =\bar\tau^n_0(x\cdot y)\star_1 1 =_{(A1)}(\bar\tau^n_0(x)\cdot \bar\tau^n_0(y))\star_1 1 =_{L.\ref{lem:emme2}(1)} ((\bar\tau^n_0(x)\star_1 1)\cdot \bar\tau^n_0(y))\star_1 1=(x[n]\cdot \bar\tau^n_0(y))_{\pref 1}$.
  
  \smallskip\noindent
  (iv) $(x\cdot y)[n] = \bar\tau^n_0(x\cdot y)\star_1 1 =_{(A1)}(\bar\tau^n_0(x)\cdot \bar\tau^n_0(y))\star_1 1=_{(A2)}(\bar\tau^n_0(x)\star_1 1)\cdot(\bar\tau^n_0(y)\star_1 1)=x[n]\cdot y[n]$.
  \end{proof}
We conclude this section with three  examples. 

\begin{example}\label{exa:ar}
   Let $\mathbb N^\star=\mathbb N\setminus\{0\}$, and let $\mathbf N_m= (\mathbb N^\star,\times,1)$ be the multiplicative monoid of positive natural numbers.
 Let  $p_0,p_1,\dots,p_k,\dots$ be the sequence of prime numbers;  for every number $n\geq 1$,   $n=\prod_{i\in\omega} p_i^{n_i}$ is the factorisation of $n$ in prime numbers. The set $\mathbb N^\star$ can be endowed with a structure of merge algebra as follows, for all $n,m\in \mathbb N^\star$: 
 \begin{itemize} 
\item $n\star_k^{\mathbb N^\star} m = (\prod_{i=0}^{k-1}p_i^{n_i})\times( \prod_{j\geq k} p_j^{m_j})$;
\item $\bar\sigma^{\mathbb N^\star}(n)=\prod_{i\in\omega} p_i^{n_{\sigma_i}}$.
\end{itemize}
The pair $\mathbf N^\star=(\mathbf N^\star_0, \mathbf N^\star_1)$, where $\mathbf N^\star_0=\mathbf N_m$ and $\mathbf N^\star_1=(\mathbb N^\star,\star_k^{\mathbb N^\star},\bar\sigma^{\mathbb N^\star} )$, is a strong am-monoid.
\end{example}

\begin{example} \label{sec:m-mon-merge-alg}
Let $M$ be a set of cardinality $\geq 2$ and $\mathcal M=(\mathbf M_i: i\in\omega)$ be a family of monoids $\mathbf M_i=(M,\times_i,\epsilon_i)$ that have the same universe $M$. Let $\prod_{i\in\omega}\mathbf M_i = (M^\omega,\cdot,1)$ be the product of the family $\mathcal M$ of monoids, where $1=(\epsilon_i: i\in\omega)$.

We denote by  $\mathcal M(\omega)$ the pair $(\mathcal M(\omega)_0,\mathcal M(\omega)_1)$ such that $\mathcal M(\omega)_0=\prod_{i\in\omega}\mathbf M_i$ and $\mathbf M(\omega)_1$ is the fully canonical merge algebra $\mathbf{Seq}(M)$ on $M$. It can be proved that $\mathcal M(\omega)$, called the \emph{product m-monoid of the family $\mathcal M$}, is an extensional m-monoid that satisfies the left distributivity law $x\cdot(y\star_n z)=(x\cdot y)\star_n (x\cdot z)$. 

It is standard to prove that $\mathcal M(\omega)$ is never a cm-monoid, and that it is an am-monoid if and only if $\mathbf M_i =\mathbf M_j$ for every $i,j\in\omega$. In this latter case $\mathcal M(\omega)$ is a strong am-monoid.

The strong am-monoid $\mathbf N^\star=(\mathbf N^\star_0, \mathbf N^\star_1)$ of natural numbers defined in Example \ref{exa:ar} is isomorphic to a product m-monoid.
Let $\mathbf N_a= (\mathbb N,+,0)$ be the additive monoid of natural numbers.  The correspondence $n\mapsto (n_0,\ldots,n_k,\ldots)$, where $n=\prod_{i\in\omega} p_i^{n_i}$ is the factorisation of $n$ in prime numbers,  embeds $\mathbf N^\star$ into the strong am-monoid $\mathcal N_a(\omega)$, which is the product m-monoid of the family $\mathcal N_a=(\mathbf N_a: i\in \omega)$. The image of this embedding is $[(0,0,\ldots,0,\ldots)]_{\equiv \omega}$.
\end{example}

Notice that, in Example \ref{exa:dsum}, if   $\mathbf M$ is an
  am-monoid (resp. a cm-monoid), then $\mathbf M\oplus \mathbf M$ is also an am-monoid (resp. a cm-monoid).

\begin{example} Every pointed   merge algebra $(\mathbf A,1^\mathbf A)$  determines an m-monoid, denoted by $m(\mathbf A)=(m(\mathbf A)_0, m(\mathbf A)_1)$, where $m(\mathbf A)_1= \mathbf A$ and $m(\mathbf A)_0=(A,\cdot,1^\mathbf A)$ is  an idempotent monoid, whose multiplication is defined as follows:
$$a\cdot b =\begin{cases} b&\text{if  $b\neq 1^\mathbf A$}\\ a &\text{if  $b= 1^\mathbf A$.} \end{cases}$$
The m-monoid $m(\mathbf A)$ will be called the \emph{left-zero m-monoid} of $\mathbf A$. It is easy to show that
 $m(\mathbf A)$ is a cm-monoid if and only if  $\mathbf A$ is degenerate, and that
 $m(\mathbf A)$ is an am-monoid if and only if $\bar\sigma(1)= 1$ for every permutation $\sigma$.

The following is an example of a left-zero m-monoid that is neither a cm-monoid nor an am-monoid.
 Let $X$ be a set of cardinality $2$ and let $(\mathbf{Seq}(X),1^{\mathrm{seq}})$ be the full canonical merge algebra on $X$ pointed at some  $1^{\mathrm{seq}}=(\epsilon_0,\epsilon_1,\dots,\epsilon_n,\dots)$ such that $\epsilon_0\neq\epsilon_1$. Then, the m-monoid $m(\mathbf{Seq}(X))$ is neither a cm-monoid nor an am-monoid, because  it is not degenerate and  $\bar\tau^1_0(1^{\mathrm{seq}})=(\epsilon_1,\epsilon_0,\epsilon_2\dots,\epsilon_n,\dots)\neq 1^{\mathrm{seq}}$.
 \end{example}

\section{Dimension}\label{sec:dim}

In this section, we aim to lift the notion of finite dimensionality from clone algebras to m-monoids.

Let $\mathbf M$ be an m-monoid, and let $D(a,n,m)$ and $D_\omega(a,n,m)$ be formulas defined as follows:
  $$   D(a,n,m)::= \forall b \forall k[ (a\cdot (1\star_m b\star_{m+k} 1))\star_n a=a];\quad\quad
    D_\omega(a,n,m)::= \forall b[ (a\cdot (1\star_m b))\star_n a=a]$$
where $a,b$ ranges over $M$, and  $m,n,k$ over the natural numbers.

\begin{definition} Let $\mathbf M$ be an m-monoid and $a\in M$.
\begin{itemize}
\item[(i)] The element $a$ is \emph{finite-dimensional}  if  $\mathbf M\models \forall n\exists m\, D(a,n,m)$.
\item[(ii)]  
  The element $a$ is \emph{$\omega$-finite-dimensional} if  $\mathbf M\models \forall n\exists m\, D_\omega(a,n,m)$.
\item[(iii)]   $\mathbf M$ is \emph{($\omega$-)finite-dimensional} if each element of $M$ is such. 
\end{itemize}
\end{definition}

If $\mathbf M$ is an m-monoid, we denote by $ M_{\mathrm{fd}}$
($ M_{\omega\mathrm{fd}}$) the set of ($\omega$-)finite-dimensional elements of $M$.

The following example unrolls this definition in a paradigmatic case.

\begin{example}\label{exa:fd}
  The cm-monoid $\mathbf F_X^{(\omega)}$ of all endofunctions of $X^\omega$, presented in Example \ref{exa:funct}, has as elements the functions from $X^\omega$ to $X^\omega$.
Recall that a function $\varphi:X^\omega\rightarrow X^\omega$ may be seen as  a sequence $\varphi=(\varphi_0,\ldots,\varphi_n,\ldots)$ of  $\omega$-ary functions $\varphi_i:X^\omega \rightarrow X$. Then we have:

(1) a map $\varphi:X^\omega\to X^\omega$  satisfies $D_\omega(\varphi,n,m)$
  iff for every $i\in[0,n-1)$, there exists a function $f_i$ of arity $m$ such that $\varphi_i= f_i^\top$;

    (2) a map $\varphi:X^\omega\to X^\omega$  satisfies $D(\varphi,n,m)$ iff
 for every $i\in[0,n-1)$ and $s\in X^\omega$, there exists a function $f_{i,s}$ of arity $m$ such that $(\varphi_i)_{|[s]_\equiv}=f_{i,s}^\top$, where $[s]_\equiv$ is defined in Section \ref{sec:tra}.
  \end{example}

  \begin{remark} ({\em{Finitely ranked versus ($\omega$-)finite-dimensional}})
The $\omega$-finite (resp. finite) dimensionality property states that, for every $i$, $\varphi_i$ is  the top extension of a finitary function (resp. is locally 
the top extension of a finitary function), whereas finite rankedness refers to the fact that there exists $j$ such that
$\varphi=(\varphi_0,\ldots,\varphi_{j-1},\e_j,\e_{j+1}\ldots,\e_{j+n},\ldots)$.
Thus, finite rankedness is a property of the whole sequence $\varphi$, whereas ($\omega$-)finite dimensionality refers to each $\varphi_i$.
\end{remark}

\begin{lemma}\label{lem:D(a,n)}  Let $\mathbf M$ be an m-monoid, $a\in M$ and $n,m\in\omega$. 
\begin{enumerate}
\item If $\mathbf M\models D(a,n,m)$, then $\mathbf M\models (\forall n'\leq n) D(a,n',m)$.
\item If $\mathbf M\models D(a,n,m)$, then $\mathbf M\models (\forall m'\geq m) D(a,n,m')$.
\end{enumerate}
The same results hold for $D_\omega$.

\end{lemma}

\begin{proof}
(1) Let $k\in\omega$ and let $b\in M$ such that $\mathbf M\models (a\cdot (1\star_m b\star_{m+k}1))\star_n a=a$. Fix $n'\leq n$. From $(a\cdot (1\star_m b\star_{m+k}1))\star_n a=a$ it follows that 
$(a\cdot (1\star_m b\star_{m+k}1))\star_{n'} a=_{(B3)}[(a\cdot (1\star_m b\star_{m+k}1))\star_n a]\star_{n'} a=_{hp}a\star_{n'}a=a$. 

(2) Let $m'\geq m$ and let $k'\in\omega$ and $b'\in M$. We want to prove that  $(a\cdot (1\star_{m'} b'\star_{m'+k'}1))\star_n a=a$.
We have $1\star_{m'} b'\star_{m'+k'}1=1\star_m(1\star_{m'} b')\star_{m+(m'-m)+k'}1$. We then instantiate $D(a,n,m)::= \forall b \forall k[ (a\cdot (1\star_m b\star_{m+k} 1))\star_n a=a]$ with $b=1\star_{m'} b'$ and $k=(m'-m)+k'$ to get 
$a=(a\cdot (1\star_m (1\star_{m'} b')\star_{m+(m'-m)+k'} 1))\star_n a= (a\cdot (1\star_m 1\star_{m'} b'\star_{m'+k'} 1))\star_n a=(a\cdot (1\star_{m'} b'\star_{m'+k'}1))\star_n a$. Hence, $D(a,n,m')$  holds.
\end{proof}

\begin{lemma}\label{lem:omega}
Let $\mathbf M$ be an m-monoid. Then $ M_{\omega\mathrm{fd}}\subseteq M_{\mathrm{fd}}$.
\end{lemma}
\begin{proof}
The set $\{1\star_m b: b\in M\}$ includes $\{1\star_m b\star_{m+k} 1: b\in M, k\in\omega\}$. 
\end{proof}

\begin{example} The inclusion $ M_{\omega\mathrm{fd}}\subseteq M_{\mathrm{fd}}$ can be strict. For example,
the function $\varphi:2^\omega\to 2^\omega$, defined by 
$\varphi(s)=1^\omega$ if $|\{ s_i : s_i=1\}|<\omega$; $0^\omega$ otherwise, is finite-dimensional but not $\omega$-finite-dimensional in the cm-monoid $\mathbf F^{(\omega)}_2$.
\end{example}

\begin{lemma}\label{lem:frfin}
 Let $\mathbf M$ be a finitely ranked m-monoid. Then $M_{\omega\mathrm{fd}}= M_{\mathrm{fd}}$.
\end{lemma}

\begin{proof}
 Let $a\in M$ be finitely dimensional, and fix $n\in\omega$. Then, there exists $m$ such that $(a\cdot (1\star_m b\star_{m+k} 1))\star_n a=a$ for every $k$ and $b$.  We now show that $(a\cdot (1\star_m b))\star_n a=a$ for every $b$. Since $b$ has finite rank, there exists $l$ such that $b\star_r 1=b$ for every $r\geq l$.  For $r\geq \mathrm{max}\{l,m\}$ we have: 
 $1\star_m b = 1\star_m b\star_r 1= 1\star_m b\star_{m+(r-m)} 1$. Using this, and substituting $k=r-m$, we get:    $(a\cdot (1\star_m b))\star_n a= (a\cdot (1\star_m b\star_{m+k} 1))\star_n a=_{hp}a$, which concludes the proof.
\end{proof}

\begin{lemma}\label{lem:findim}  Let $\mathbf M$ be a cm-monoid or an am-monoid. Then $ M_{\mathrm{fd}}$ and $ M_{\omega\mathrm{fd}}$ are pointed merge subalgebras of the merge reduct of $\mathbf M$.
\end{lemma}

\begin{proof} (1) We show that $\mathbf M\models D_\omega(1,n,n)$.
Since $(M, \star_n)$ is a rectangular band, it satisfies $x \star_n y \star_n x = x$. Then we have:
$(1\cdot (1\star_n b))\star_n 1=1\star_n b\star_n 1 =1$. Then $\mathbf M\models \forall n D_\omega(1,n,n)$.
 
 (2) Let $a\in  M_{\omega\mathrm{fd}}$ and  $\rho\in S_\omega$ be a permutation such that $\rho(i)=i$ for every $i\geq k$.  Let  $n,m>k$  such that $(a\cdot (1\star_m b))\star_n a=a$ for every $b\in M$. 
Then we have: $\bar\rho(a)=\bar\rho((a\cdot (1\star_m b))\star_n a)=_{L.\ref{lem:b8b12}(4)}\bar\rho(a\cdot (1\star_m b))\star_n \bar\rho(a)$.

If $\mathbf M$ is a cm-monoid then $\bar\rho(a)=\bar\rho(a\cdot (1\star_m b))\star_n \bar\rho(a)=_{(LS)}(\bar\rho(a)\cdot (1\star_m b))\star_n \bar\rho(a)$.

If $\mathbf M$ is a am-monoid then $\bar\rho(a)=\bar\rho(a\cdot (1\star_m b))\star_n \bar\rho(a)=_{(A1)}(\bar\rho(a)\cdot \bar\rho(1\star_m b))\star_n \bar\rho(a)
=_{L.\ref{lem:b8b12}(1)}
(\bar\rho(a)\cdot (\bar\rho(1)\star_m   b))\star_n \bar\rho(a)=
(\bar\rho(a)\cdot (1\star_m b))\star_n \bar\rho(a)
$, because, by Lemma \ref{lem:am},  $\bar\rho(1)=1$ for every $\rho$ in every am-monoid.

By applying Lemma \ref{lem:D(a,n)} we conclude $\mathbf M\models \forall n\exists mD_\omega(\bar\rho(a),n,m)$.

 (3) Let  $n, m\in \omega$ such that $(a\cdot (1\star_m c))\star_n a=a$ and $(b\cdot (1\star_m c))\star_n b=b$ for every $c\in M$. 
 Without loss of generality, by Lemma \ref{lem:D(a,n)} we can assume $n>k$. Then we have:
 $((a\star_k b)\cdot (1\star_m c))\star_n (a\star_k b)=_{(RD)}
 (a\cdot (1\star_m c))\star_k  (b\cdot(1\star_m c))\star_n (a\star_k b)=_{(B3)} (a\cdot (1\star_m c))\star_k  (b\cdot(1\star_m c))\star_n b
=_{(B4)} (a\cdot (1\star_m c))\star_k  [(b\cdot(1\star_m c))\star_n b] =_{hp} (a\cdot (1\star_m c))\star_k b =_{(B3)} ((a\cdot (1\star_m c))\star_n a) \star_k b=a \star_k b$.

The conclusion follows for $ M_{\omega\mathrm{fd}}$. A similar proof works for $ M_{\mathrm{fd}}$.
\end{proof}

From Lemma \ref{lem:findim} it follows that, if $a$ is ($\omega$-)finite-dimensional, then $a[i]$ is also ($\omega$-)finite-dimensional for every natural number $i$. In particular, the coordinates of the unit $1$ are $\omega$-finite-dimensional.

\begin{lemma}\label{lem:mah}
Let $\mathbf M$ be an extensional cm-monoid, $a\in M$ and $n,m\in\omega$. We have that 
$\mathbf M\models D(a,n,m)$ if and only if  
$\mathbf M\models \forall i\in[0,n)\forall b\forall k ((a[i]\cdot (1\star_m b\star_{m+k}1))_{<1} = a[i])$. A similar result holds for $D_\omega(a,n,m)$.
\end{lemma}

\begin{proof}  
By extensionality and Lemmas \ref{lem:[k]bis}(1) and \ref{lem:[k]}, we have that the following conditions are equivalent:
 \begin{enumerate}
\item[(i)] $(a\cdot (1\star_m b\star_{m+k}1))\star_n a=a$;
\item[(ii)] $a$ and $(a\cdot (1\star_k b\star_{k+m}1))\star_n a$ have the same coordinates;
\item[(iii)] $a[i]=(a[i]\cdot (1\star_k b\star_{k+m}1))_{<1}$ for every $0\leq i\leq n-1$.
\end{enumerate}
  \end{proof}

\section{Relating cm-monoids and clone algebras} \label{sec:ca1}

In this section, we relate  the category of clone algebras to the category of cm-monoids via an adjunction, which restricts to an equivalence when considering  finitely ranked cm-monoids.

We first show how to associate  functorially a cm-monoid to a clone algebra, and vice-versa, and then prove that these constructions are adjoint to each other.

Let $\mathbf C=(C,q_n^\mathbf C, \e_n^\mathbf C)$ be a  clone algebra and $1=(\e_0^\mathbf C,\dots,\e_n^\mathbf C,\dots)$.
We define a cm-monoid $\mathbf C^{\mathrm{cm}}=(\mathbf C^{\mathrm{cm}}_0,\mathbf C^{\mathrm{cm}}_1)$,
where $\mathbf C^{\mathrm{cm}}_1$ is the canonical merge algebra of universe $[1]_\equiv\subseteq C^\omega$. The monoid $\mathbf C^{\mathrm{cm}}_0=([1]_\equiv,\cdot^{\mathrm{cm}},1)$ is defined as follows, for every $ a,  b\in [1]_\equiv$ such that $a_i = \e_i^\mathbf C$ for every $i\geq k$: 
$$b\cdot^{\mathrm{cm}} a = (q_k^\mathbf C(b_0, a_0,\dots,a_{k-1}),\dots,q_k^\mathbf C(b_n, a_0,\dots,a_{k-1}),\dots).$$
 Note that  $ b\cdot^{\mathrm{cm}} a$ is well defined by (C4), and that
 $b\cdot^{\mathrm{cm}} a\in [1]_\equiv$ because there exists $n\geq k$ such that $b_j=\e_j^\mathbf C$ and $q_k^\mathbf C(b_j, a_0,\dots,a_{k-1}) = \e_j^\mathbf C$ for every $j\geq n$.   

\begin{proposition} \label{th:cm-functor-correct}
Let $\mathbf C$ be a  clone algebra. Then the following conditions hold:
\begin{enumerate}
\item  The algebra $\mathbf C^{\mathrm{cm}}$ is a canonical and finitely ranked cm-monoid. 
\item If $\mathbf C$ is finite-dimensional, then $\mathbf C^{\mathrm{cm}}$ is finite-dimensional too.
\end{enumerate}
\end{proposition}

\begin{proof} (1) The algebra $\mathbf C^{\mathrm{cm}}_1$ is a canonical merge algebra by definition.
We now prove that $\mathbf C^{\mathrm{cm}}_0$ is a monoid and that  (RD) holds.
Let $x,y,z\in [1]_\equiv$ such that $x_i=y_i=z_i=\e_i$ for every $i\geq k$.
\begin{enumerate}
\item[(i)]  $ x \cdot 1=(q_k(x_0,\e_0,\dots,\e_{k-1}),\dots,q_k(x_n,\e_0,\dots,\e_{k-1}),\dots)=_{(C3)}  x$.
\item[(ii)] $1 \cdot  x=(q_k(\e_0, x_0,\dots,x_{k-1}),\dots,q_k(\e_n, x_0,\dots,x_{k-1}),\dots)=_{(C1,C2)}  x$, because $x_i=\e_i$ for  $i\geq k$.
\item[(iii)] 
\[
\begin{array}{lll}
 ( x \cdot  y)\cdot z   &  = & (q_k(x_i, y_0,\dots,y_{k-1}):i\in\omega) \cdot  z \\
  &  = &  (q_k(q_k(x_i, y_0,\dots,y_{k-1}), z_0,\dots,z_{k-1}):i\in\omega) \\
   &  =_{C5} & (q_k(x_i,q_k(y_0,z_0,\dots,z_{k-1}),\dots,q_k(y_{k-1},z_0,\dots,z_{k-1}),\e_k,\dots) :i\in\omega) \\
      &  = & x\cdot (q_k(y_0,z_0,\dots,z_{k-1}),\dots,q_k(y_{k-1},z_0,\dots,z_{k-1}),\e_k,\dots) \\
     &  = &  x \cdot ( y\cdot  z).  \\
\end{array}
\]

\item[(iv)] If $\bar z= z_0,\dots,z_{k-1}$, then
$(x\star_n y)\cdot z=(q_k(x_0, \bar z),\dots,q_k(x_{n-1}, \bar z),q_k(y_n, \bar z),\dots)= (x\cdot z) \star_n( y \cdot z)$.
\end{enumerate}
It is easy to see that the m-monoid $\mathbf C^{\mathrm{cm}}$ is finitely ranked. 
From canonicity and from Proposition \ref{prop:loc=str}  it follows that $\mathbf C^{\mathrm{cm}}$ is a cm-monoid. 

\medskip

\noindent (2)  Let $1=(\e_i^\mathbf C:i\in\omega)$,  $C^{\mathrm{cm}}=[1]_\equiv$ and $a=(a_0,\dots,a_{m-1},\e_m,\e_{m+1},\dots)\in [1]_\equiv$. We have to show that $a$ is finite-dimensional in the canonical cm-monoid $\mathbf C^{\mathrm{cm}}$. By Lemma \ref{lem:mah} this means that, for every $n$ there exists $k$ such that
 $\forall i\in[0,n)\forall r\forall b\ (a[i]\cdot^{\mathrm{cm}} (1\star_k^{\mathrm{cm}} b\star_{k+r}^{\mathrm{cm}} 1))_{<1} = a[i]$.
  Given $n\in \omega$,  we choose $k$ to be a natural number greater than $n$, $m$ and $\mathrm{dim}(a_0),\dots, \mathrm{dim}(a_{m-1})$ (see Section \ref{sec:clonealg}). Let $c> k+r$ such that $a_i=\e_i$ for every $i\geq c$. Then we have:
\begin{itemize}
\item $(a\cdot (1\star_k b\star_{k+r} 1))_i= q_c(a_i,\e_0,\dots,\e_{k-1},b_k,\dots,b_{k+r-1},\e_{k+r},\dots)=q_k(a_i,\e_0,\dots,\e_{k-1})=a_i$ for every $0\leq i\leq m-1$, because $\mathrm{dim}(a_i)<k$;
\item $(a\cdot (1\star_k b\star_{k+r} 1))_i=\e_{i}$ ($m\leq i\leq k-1$);
\end{itemize}
It follows that $(a\cdot (1\star_k b\star_{k+r} 1))\star_n a=a$, because $n<k$.  Hence, $a$ is finite-dimensional in $\mathbf C^{\mathrm{cm}}$.
\end{proof}

Note that by Lemma \ref{lem:coto}  $\mathbf C^{\mathrm{cm}}$ is extensional. 

\medskip
We now go in the reverse direction and show how to get clone algebras from cm-monoids.
Let $\mathbf M$ be a  cm-monoid
and  $a=(a_i\in M_{|i+1} :i\in\omega)$ be a sequence of elements. We define by induction a sequence of elements of $M$ as follows:
$$\widehat a_0 = a_0;\qquad \widehat a_n=\widehat a_{n-1}  \star_n \bar\tau^n_0 ( a_n).$$
We shall need this notation in fact only for sequences in which all $a_i$ are in
$M_{|1}$, but the above definition places this construction in its right context
of generality. We shall also apply the construction to finite sequences $\bar b=(b_0,\dots,b_{m-1})$,
setting $b=(b_0,\dots,b_{m-1},1[m],1[m+1],\dots)$ and $\widehat{\bar b}_m=\widehat{b}_m$.

\begin{lemma}\label{lem:8.3}
 $(\widehat a_n)_{\pref n+1}= \widehat a_n$.
\end{lemma}

\begin{proof}
We get the conclusion as follows:
 $(\widehat a_n)_{\pref n+1}= (\widehat a_{n-1}  \star_n \bar\tau^n_0 ( a_n))\star_{n+1} 1 =_{(B4)} \widehat a_{n-1}  \star_n (\bar\tau^n_0 ( a_n)\star_{n+1} 1)=_{L.\ref{lem:b8b12}(1)}\widehat a_{n-1}  \star_n \bar\tau^n_0 ( a_n\star_{n+1} 1)=\widehat a_{n-1}  \star_n \bar\tau^n_0 ( a_n)= \widehat a_n$, because $a_n\in M_{|n+1}$.
\end{proof}

It follows from Lemma \ref{lem:8.3}  that $\widehat a_n\in M_{|n+1}$.

\begin{lemma}\label{lem:bar} Let $ a=(a_i\in M_{|i+1} :i\in\omega)$. Then we have:
$\widehat a_n[k]=\begin{cases}(a_k)_{\pref 1}&\text{if $0\leq k\leq n$}\\
  1[k] &\text{otherwise}\\\end{cases}$.
\end{lemma}

\begin{proof}
 The proof is by induction on $n$. 
 
 ($n=0$): By Lemma \ref{lem:[k]} we have $\widehat a_0[0]=a_0 = (a_0)_{\pref 1}$ and $\widehat a_0[k]=1[k]$ for $k>0$.
 
 ($n>0$): By Lemma \ref{lem:[k]}, since $\widehat a_n[k]=\widehat a_{n-1}  \star_n \bar\tau^n_0 ( a_n)$, then  for  $0\leq k\leq n-1$ we have that $\widehat a_n[k] =\widehat a_{n-1}[k]$ and, for $k\geq n$,  $\widehat a_n[k]=\bar\tau^n_0 ( a_n)[k]$.
By  induction hypothesis
$\widehat a_n[k] =\widehat a_{n-1}[k]=(a_k)_{\pref 1}$ for $0\leq k\leq n-1$. Moreover, for $k=n$,   $\widehat a_n[n] = \bar\tau^n_0( a_n)[n]=_{\mathrm{L.}\ref{lem:[k]}(i)}a_n[\tau^n_0(n)]=a_n[0]=(a_n)_{\pref 1}$ and, for every $k>n$, $\widehat a_n[k] = \tau^n_0 ( a_n)[k]=a_n[\tau^n_0(k)]=a_n[k]=1[k]$, because $a_n$ has rank $n+1$.
\end{proof}

\begin{lemma}\label{lem:a=hat}
 Let $a\in M$ and $a_\sq=(a[0],a[1],\dots)$ be the sequence of its coordinates. Then $a\star_{n+1} 1 = (\widehat{a_\sq})_n$. Moreover, if $a$ has rank $n+1$, then $a = (\widehat{a_\sq})_k$ for all $k\geq n$.
\end{lemma}

\begin{proof}
 By the hypothesis and  Lemma \ref{lem:bar} we have that $a\star_{n+1} 1$ and $(\widehat{a_\sq})_n$ have the same coordinates. 
 The conclusion follows from Lemma \ref{lem:B12}(3).
 Moreover,  if $a$ has rank $n+1$, then $a$ and  $(\widehat{a_\sq})_k$ have the same coordinates, for all $k\geq n$. The conclusion again follows from Lemma  \ref{lem:B12}(3).
\end{proof}

  Let $\mathbf M$ be a cm-monoid. We define an algebra $\mathbf M^{\mathrm{ca}}=(M_{|1},q_n^{\mathrm{ca}},$ $\e_n^{\mathrm{ca}})_{n\geq 0}$ in the similarity type of clone algebras as follows:
  $q_n^{\mathrm{ca}}(a,b_0,\dots,b_{n-1})=(a\cdot^\mathbf M \widehat b_{n-1})_{\pref 1};\quad \e_n^{\mathrm{ca}} = 1^\mathbf M[n]$,
using the notation introduced above.

\begin{lemma}\label{lem:1hat}  Let $\mathbf M$ be a cm-monoid and  $E=(1[k]:k\in\omega)$. Then $\widehat E_k = 1$ for every $k$.
\end{lemma}

\begin{proof}
 The proof is by induction:  $\widehat E_k=\widehat E_{k-1}  \star_k \bar\tau^k_0 ( 1[k])=_{Ind.} 1  \star_k \bar\tau^k_0 ( 1[k])$\\ $=_{L. \ref{lem:B12}(4)} 1  \star_k 1[k] = 1\star_k (\bar\tau^k_0(1)\star_1 1)=_{(B3)} 1\star_k 1=1$.
\end{proof}

\begin{proposition} \label{th:ca-functor-correct}
Let $\mathbf M$ be a cm-monoid.
 Then the following conditions hold: 
 \begin{enumerate}
\item The algebra $\mathbf M^{\mathrm{ca}}$ is a  clone algebra.
\item If $\mathbf M$ is finite-dimensional, then $\mathbf M^{\mathrm{ca}}$ is finite-dimensional. 
\end{enumerate}
 \end{proposition}

\begin{proof} (1) We prove the identities defining the variety of clone algebras.
\begin{itemize}
\item[(C1-C2)] Since $\mathbf M$ is a cm-monoid, then by Lemma \ref{lem:[k]bis}(2) $a[k]=  (1[k]\cdot a)_{\pref 1}$.
Then by Lemma \ref{lem:bar} we have:\\
 $q_n(1[k],b_0,\dots,b_{n-1})=(1[k]\cdot \widehat b_{n-1})_{\pref 1}= \widehat b_{n-1}[k]=\begin{cases}b_k&\text{if $0\leq k\leq n-1$}\\
 1[k]&\text{otherwise.}\end{cases}$
\item[(C3)]  $q_n(a,1[0],\dots,1[n-1])= (a\cdot \widehat E_{n-1})_{\pref 1}=_{L.\ref{lem:1hat}}(a\cdot 1)_{\pref 1}=a_{\pref 1}=a$.

\item[(C4)] Let $c\in (M_{|1})^\omega$ be the sequence $c=(b_0,\dots,b_{n-1},1[n],1[n+1],\dots)$. It is easy to prove that
$$\widehat c_k=\begin{cases}\widehat b_k&\text{if $0\leq k\leq n-1$}\\
  \widehat b_{n-1}&\text{if $k\geq n$.}\end{cases}$$
    Indeed, we prove by induction that $\widehat c_k=\widehat b_{n-1}$ for all $k\geq n$.
  The base case is as follows: 
  $\widehat c_n=\widehat c_{n-1}  \star_n \tau^n_0 ( 1[n])=_{Ind, L. \ref{lem:B12}(4)} \widehat b_{n-1}  \star_n 1[n]= \widehat b_{n-1}  \star_n (\bar\tau^n_0(1)\star_1 1)= \widehat b_{n-1}\star_n 1 =_{L.\ref{lem:8.3}} \widehat b_{n-1}$, and the inductive step is similar.
Therefore, for every $k>n$, we get:\\
$q_k(a,b_0,\dots,b_{n-1},1[n],\dots,1[k-1])= (a\cdot \widehat c_{k-1})_{\pref 1}= (a\cdot \widehat b_{n-1})_{\pref 1} = q_n(a,b_0,\dots,b_{n-1})$.
\item[(C5)] We prove that $q_n(q_n(a,\bar b),\bar c)=q_n(a,q_n(b_0,\bar c),\dots,q_n(b_{n-1},\bar c))$.
\[
\begin{array}{llll}
q_n(q_n(a,\bar b),\bar c)  &  = &  (q_n(a,\bar b)\cdot \widehat c_{n-1})_{\pref 1}& \\
  &  = &  [(a\cdot \widehat b_{n-1})_{\pref 1}\cdot  \widehat c_{n-1}]_{\pref 1} &\\
    &  = &  [a\cdot \widehat b_{n-1}\cdot  \widehat c_{n-1}]_{\pref 1} &\text{ by Lemma \ref{lem:emme2}(1)}\\
      &  = &  [a\cdot (\widehat b_{n-1}\cdot  \widehat c_{n-1})]_{\pref 1}.& \\
\end{array}
\]
Consider the sequence $d_i=(b_i\cdot \widehat c_{n-1})_{\pref 1}= q_n(b_i,\bar c)$ ($0\leq i\leq n-1$). Then we have
$q_n(a,q_n(b_0,\bar c),\dots,q_n(b_{n-1},\bar c))=(a\cdot \widehat d_{n-1})_{\pref 1}$.

\noindent We prove by induction that $\widehat d_i = (\widehat b_i\cdot \widehat c_{n-1})_{\pref i+1}$. The case $i=0$ is trivial.
\[
\begin{array}{llll}
(\widehat b_i\cdot \widehat c_{n-1})_{\pref i+1}&  = &((\widehat b_{i-1} \star_i \bar\tau^i_0(b_i))\cdot \widehat c_{n-1})\star_{i+1} 1&\\
&  = & (\widehat b_{i-1}\cdot \widehat c_{n-1})\star_i (\bar\tau^i_0(b_i)\cdot \widehat c_{n-1})\star_{i+1} 1 & \text{(RD)} \\
&  = & (\widehat b_{i-1}\cdot \widehat c_{n-1})_{\pref i}\star_i (\bar\tau^i_0(1) \cdot b_i\cdot \widehat c_{n-1})\star_{i+1} 1 & \text{(B1) and (LS)} \\
&  = & \widehat d_{i-1}\star_i \bar\tau^i_0(b_i\cdot \widehat c_{n-1})\star_{i+1} 1&\text{induction and (LS)}\\
&=& \widehat d_{i-1}\star_i (\bar\tau^i_0(b_i\cdot \widehat c_{n-1})\star_{i+1} 1)& \text{(B4)} \\
&=& \widehat d_{i-1}\star_i \bar\tau^i_0((b_i\cdot \widehat c_{n-1})\star_{i+1} 1)& \text{Lemma \ref{lem:b8b12}(1)} \\
&=& \widehat d_{i-1}\star_i \bar\tau^i_0((b_i\cdot \widehat c_{n-1})\star_1 1)& \text{Lemma \ref{lem:B12}(5)} \\
&=& \widehat d_{i-1}\star_i \bar\tau^i_0(d_i)&\\
&=& \widehat d_i.
\end{array}
\]
\end{itemize}
By Lemmas \ref{lem:emme2} and \ref{lem:8.3} we conclude: $\widehat d_{n-1} = (\widehat b_{n-1}\cdot \widehat c_{n-1})_{\pref n}= \widehat b_{n-1}\cdot \widehat c_{n-1}$.

(2) 
Let $a\in M_{|1}$ be finite-dimensional and $n=1$. Then there exists $m$ such that  for every $c\in M$ and $k\in\omega$, we have $a=(a\cdot (1\star_m c\star_{m+k}1))\star_1 a$.
Now we show that $a$ is independent of $\e_l$ for every $l\geq m$ in the clone algebra $\mathbf M^{\mathrm{ca}}$. In other words, we show that $q_{l+1}^{\mathrm{ca}}(a,\e_0,\dots,\e_{l-1},b)=a$ for every $l\geq m$ and $b\in M_{|1}$.
Consider the sequence $\hat b=(1[0],1[1],\dots,1[l-1], b)$.  For every $k\geq 1$, we have:
\begin{equation}\label{eq:cappuccio}\hat b_l =_{\mathrm{def}} 
\widehat b_{l-1}  \star_l \bar\tau^l_0 (b)=_{L.\ref{lem:1hat}}
1  \star_l \bar\tau^l_0 (b)=1  \star_l \bar\tau^l_0 (b)\star_{l+k} 1.
\end{equation}
The last equality above holds because $b$ is a fortiori of rank $\leq  l+k$, so that $b\star_{l+k} 1= b$, and $\bar\tau^l_0 (b)=\bar\tau^l_0 (b\star_{l+k} 1)=_{L.\ref{lem:b8b12}(1)} \bar\tau^l_0 (b)\star_{l+k} 1$.
Thus, 
\[
\begin{array}{llll}
      q_{l+1}^{\mathrm{ca}}(a,\e_0,\dots,\e_{l-1},b)
   &=  &  (a\cdot \widehat b_l)\star_1 1&\text{by def.}   \\
 &  =  & (a\cdot (1\star_l \bar\tau^l_0 (b)\star_{l+k} 1))\star_1 1  &
   \text{by (\ref{eq:cappuccio})}\\
 &   =&(a\cdot (1\star_l \bar\tau^l_0 (b)\star_{l+k} 1))\star_1 a\star_1 1&\text{by (B1)}\\ 
 & = & (a\cdot (1\star_l \bar\tau^l_0 (b)\star_{l+k} 1))\star_1 a    &\text{because $a$ has rank $\leq 1$}\\
 &   =&a&\text{by Lemma \ref{lem:D(a,n)}(2)}.\\
\end{array}
\]
\end{proof}

\begin{lemma}\label{rem:can}
  Let  $\mathbf M$ be a canonical cm-monoid with $1=(1_0,\ldots,1_n,\ldots)$. If
  $a=(a^i\in M_{|1} : i\in\omega)$, where $a^i=(a^i_0,1_1,\ldots,1_n,\ldots)$,
  then $(\widehat a )_n=(a^0_0,a^1_0,\ldots,a^n_0,1_{n+1},\ldots)$, for all $n$.
\end{lemma}

\begin{proof} The proof is by induction over $n$.
For $n=0$ the conclusion is immediate: $\widehat a_0 = a^0 =(a^0_0,1_1,\dots,1_n,\dots)$.
Assuming that the equality is true for $n$, we provide the proof for $n+1$:    $\widehat a_{n+1} = \widehat a_n  \star_{n+1} \bar\tau_0^{n+1} ( a^{n+1})=(a^0_0,a^1_0,\ldots,a^n_0,1_{n+1},\ldots)\star_{n+1}(1_{n+1},1_1,1_2,\dots,1_n,a^{n+1}_0,1_{n+2},\dots)=
  (a^0_0,a^1_0,\dots,a^n_0,a^{n+1}_0,1_{n+2},\dots)$. 
\end{proof}

We next examine how ``inverse" the constructions $(\_)^{\mathrm{cm}}$ and $(\_)^{\mathrm{ca}}$ are.
Let $\mathbf C=(C,q_n^\mathbf C,\e_n^\mathbf C)$ be a clone algebra. 
If $a\in C$, we denote by $\langle a\rangle$ the sequence
$(a,\e_1^\mathbf C,\dots,\e_n^\mathbf C,\ldots)$.

 \noindent  The algebra $(\mathbf C^{\mathrm{cm}})^{\mathrm{ca}}$ has universe $([1]_\equiv)_{|1}$ and  $\langle \e_n^\mathbf C\rangle$ as  $n$-th projection.
 The operator $q_n$ is defined as follows for every  $\langle a  \rangle, \langle b_i   \rangle \in ([1]_\equiv)_{|1}$, and $b=( \langle b_0   \rangle,\ldots,
    \langle b_{n-1}  \rangle, \langle \e_n^{\mathbf C}\rangle, \ldots)$:
  $$ \begin{array}{lll}
  &&q_n^{\mathrm{ca}}(\langle a\rangle,\langle b_0\rangle,\dots,\langle b_{n-1}\rangle)\\
       &=& (\langle a\rangle\cdot^{\mathrm{cm}}(\widehat{b})_{n-1})_{\pref 1}\quad \text{by Lemma \ref{rem:can}}\\
   &=& (q_n^\mathbf C(a,b_0,\dots,b_{n-1} ),q_n^\mathbf C(\e_1^\mathbf C  ,b_0,\dots,b_{n-1} ),\ldots, ,q_n^\mathbf C(\e_k^\mathbf C  ,b_0,\dots,b_{n-1} ),\ldots)_{\pref 1}  \\
&=& (q_n^\mathbf C(a,b_0,\dots,b_{n-1} ),b_1,\ldots,b_{n-1}, \e_n^\mathbf C,\ldots)_{\pref 1} \\ &=& \langle q_n^\mathbf C(a, b_0,\dots,b_{n-1})\rangle
  \end{array}$$
  We trivially have $\mathbf C \cong (\mathbf C^{\mathrm{cm}})^{\mathrm{ca}}$, because the map $a\mapsto \langle a\rangle$ is a bijection from $C$ onto $([1]_\equiv)_{|1}$ that preserves all the given operations.

  Conversely, let $\mathbf M$ be a cm-monoid.
The algebra $(\mathbf M^{\mathrm{ca}})^{\mathrm{cm}}$ has  the trace $[1_{\sq}]_\equiv$ on $M_{|1}$ as universe and, for every $a= (a_0,\dots,a_{k-1},1[k],\dots)\in [1_{\sq}]_\equiv$ and $b=(b_0,\dots,b_{n-1},1[n],\dots)\in [1_{\sq}]_\equiv$ (without loss of generality we can choose $n=k$), we have: 
\[
\begin{array}{lll}
 b\cdot^{\mathrm{cm}} a  &  = &  (q_k^{\mathrm{ca}}(b_0, a_0,\dots,a_{k-1}),\dots,q_k^{\mathrm{ca}}(b_{k-1}, a_0,\dots,a_{k-1}),1[k],\dots) \\
  &  = &  ((b_0\cdot^\mathbf M \widehat a_{k-1})_{\pref 1},\dots,(b_{k-1}\cdot^\mathbf M \widehat a_{k-1})_{\pref 1},1[k],\dots) \\
    &  = &  ((\widehat b_{k-1}\cdot^\mathbf M \widehat a_{k-1})[0],\dots,(\widehat b_{k-1}\cdot^\mathbf M \widehat a_{k-1})[k-1],1[k],\dots) \\
      &  = & (\widehat b_{k-1}\cdot^\mathbf M \widehat a_{k-1})_{\sq}   \\
\end{array}
\]

 \begin{remark} 
  Let $\mathbf M$ be a  cm-monoid. It is easy to show that  $(\mathbf M^{\mathrm{ca}})^{\mathrm{cm}}=((\mathbf M_{|\omega})^{\mathrm{ca}})^{\mathrm{cm}}$.
   \end{remark}

\begin{lemma}\label{prop:iso}
Let $\mathbf M$ be a  cm-monoid. Then the map $f:M_{|\omega}\to [1_\sq]_\equiv$, defined by $f(x)=x_\sq$ for every $x\in M_{|\omega}$, is an isomorphism from $\mathbf M_{|\omega}$ onto $(\mathbf M^{\mathrm{ca}})^{\mathrm{cm}}$.
\end{lemma}

\begin{proof} We prove that $f$ is a homomorphism:
\begin{enumerate}
\item $f(x\star_n^\mathbf M y)=(x\star_n^\mathbf M y)_\sq =_{L.\ref{lem:[k]}}(x[0],\dots,x[n-1],y[n],\dots)=x_\sq \star_n^{\mathrm{cm}} y_\sq=f(x) \star_n^{\mathrm{cm}}f(y)$.

\item $f(\bar\sigma(x))=\bar\sigma(x)_\sq =_{L.\ref{lem:[k]}} (x[\sigma_i]:i\in\omega)=\bar\sigma^{\mathrm{cm}}(x_\sq)=\bar\sigma^{\mathrm{cm}}(f(x))$.

\item Assume that  $x,y$ have rank $k$. 
Then,  by  Lemma \ref{lem:bar}, $y$  and $(\widehat{y_\sq})_{k-1}$ have the same coordinates, and $y=(\widehat{y_\sq})_{k-1}$  by Lemma \ref{lem:a=hat}. Therefore, we have: 
\[
\begin{array}{lll}
 f(x) \cdot^{\mathrm{cm}} f(y)    &  = &  (x_\sq\cdot^{\mathrm{cm}} y_\sq) \\
  &  = & (q_k^{\mathrm{ca}}(x[0], y[0],\dots,y[k-1]),\dots,q_k^{\mathrm{ca}}(x[k-1], y[0],\dots,y[k-1]),1[k],\dots)  \\
  & =  & ((x[0]\cdot^\mathbf M (\widehat{y_\sq})_{k-1})_{\pref 1},\dots,(x[k-1]\cdot^\mathbf M (\widehat{y_\sq})_{k-1})_{\pref 1},1[k],\dots)  \\
& =  &((x[0]\cdot^\mathbf M y)_{\pref 1},\dots,(x[k-1]\cdot^\mathbf M y)_{\pref 1}, 1[k],\dots) \\
& =  &(x\cdot^\mathbf M y)_\sq, \qquad\text{by Lemma \ref{lem:[k]bis}(1)}\\
& =  &f(x\cdot^\mathbf M y). 
\end{array}
\]
\end{enumerate}

The homomorphism $f$ is an isomorphism since  $x\mapsto x_\sq\in [1_\sq]_\equiv$ is a bijection by Lemma \ref{lem:B12}(3).
\end{proof}

We have now almost all the ingredients to state our claimed adjunction. Let  $\mathbb{CA}$ and $\mathbb{CM}$ be the categories of clone algebras and clone algebra morphisms, and of cm-monoids and cm-monoid morphisms, respectively. We turn the constructions 
 $(\_)^{\mathrm{cm}}$ and $(\_)^{\mathrm{ca}}$ into functors $(\_)^{\mathrm{cm}}:\mathbb{CA}\rightarrow  \mathbb{CM}$ and $(\_)^{\mathrm{ca}}:\mathbb{CM}\rightarrow\mathbb{CA}$ as follows:
 \begin{itemize}
 \item Let $g:\mathbf M\rightarrow \mathbf N$ be a morphism of  cm-monoids. Then we set $g^{\mathrm{ca}}=g$, or more precisely $g^{\mathrm{ca}}$ is the restriction of $g$ to $M_{|1}$. Note that it is easy to check that, for $x\in M_{|1}$,  we have $g(x)\in N_{|1}$. 
\item Let $f:\mathbf C\rightarrow \mathbf B$ be a morphism of clone algebras. Then, for $a=(a_0,\dots,a_{n-1},\e_n^\mathbf C,\dots)\in[1]_\equiv$,  we set
$f^{\mathrm{cm}}(a)=(f(a_0),\dots,f(a_{n-1}),\e_n^\mathbf B,\dots)$.
This does not depend on the choice of the representative of $a$, since $f$ preserves all constants $\e_i$.
\end{itemize}

 \begin{theorem} \label{prop:ca-cm-adjunction} The functor  $(\_)^{\mathrm{cm}}:\mathbb{CA} \rightarrow \mathbb{CM}$ is left adjoint to the 
 functor   $(\_)^{\mathrm{ca}}:\mathbb{CM}\rightarrow\mathbb{CA}$. 
\end{theorem}

 \begin{proof}
 Let $\mathbf M$ be a cm-monoid. Then the map $g:[1_\sq]_\equiv\to M$, defined by
$g(x)=f^{-1}(x)$, where $f$ is the isomorphism defined in Lemma \ref{prop:iso},  is an embedding from 
${(\mathbf M}^{\mathrm{ca}})^{\mathrm{cm}}$ into  $\mathbf M$.
This map $g$  is a good candidate for being the counit of the adjunction, while the unit will be the isomorphism  
$\langle\_\rangle_C: \mathbf C \to  (\mathbf C^{\mathrm{cm}})^{\mathrm{ca}}$. We have to check the triangular identities of adjunction.
\begin{enumerate}
\item 
 $(g_M)^{\mathrm{ca}} \circ (\langle\_\rangle_{M^{\mathrm{ca}}}) = {\it id}$.
Let $x\in \mathbf M^{\mathrm{ca}}$, i.e., $x\in M_{|1}$. We have
$$\langle x\rangle_{M^{\mathrm{ca}}} = (x,1[1],1[2],\ldots)= x_\sq =f_M(x)\quad(\mbox{in}\;((\mathbf M^{\mathrm{ca}})^{\mathrm{cm}})^{\mathrm{ca}}),$$
and hence $(g_M)^{\mathrm{ca}} \circ (\langle\_\rangle_{M^{\mathrm{ca}}}) = g_M \circ f_M ={\it id}$.

\smallskip
\item
$( g_{C^{\mathrm{cm}}} \circ (\langle\_\rangle_C)^{\mathrm{cm}}) = {\it id}$.
Let $a \in \mathbf C^{\mathrm{cm}}$.
We have 
$(\langle a\rangle_C)^{\mathrm{cm}})=(l_0,l_1,\ldots)=a_\sq \quad(\mbox{in}\;((\mathbf C^{\mathrm{cm}})^{\mathrm{ca}})^{\mathrm{cm}})$,
where  $l_i=(a_i,\e_1,\e_2,\ldots)=a[i]$ (for all $i$).  Therefore $(\langle a\rangle_C)^{\mathrm{cm}}=a_\sq=f_{C^{\mathrm{cm}}}(a)$ (note that $a$ is in the domain of definition of $f_{C^{\mathrm{cm}}}$ since $\mathbf C^{\mathrm{cm}}$ is finitely ranked). 
We conclude as in the previous case.
\end{enumerate}
\end{proof}
\begin{corollary} \label{cor:ac-cm-adjunction} Let $\mathbb{ACL}$ be the category of abstract clones. 
The functor  $(\_)^{\mathrm{cm}}\circ(\_)^{\textrm{ac-ca}}:\mathbb{ACL} \rightarrow \mathbb{CM}$ is left adjoint to the 
 functor   $R_{(\_)}\circ (\_)^{\mathrm{ca}}:\mathbb{CM}\rightarrow\mathbb{ACL}$. 
\end{corollary}
\begin{proof}
By Theorems \ref{prop:ca-ac-adjunction} and  \ref{prop:ca-cm-adjunction}.
\end{proof}

Finally, we prove that clone algebras are categorically equivalent to finitely ranked cm-monoids, and that abstract clones are categorically equivalent to finite-dimensional and finitely ranked cm-monoids.
Let 
 $\mathbb{CM}^{\mathrm{fr}}$ be the full subcategory of $\mathbb{CM}$  whose objects are the finitely ranked cm-monoids, and let $\mathbb{CM}^{\mathrm{fdr}}$ be the full subcategory of $\mathbb{CM}^{\mathrm{fr}}$ whose objects are  the finite-dimensional and finitely ranked cm-monoids.

\begin{theorem} \label{equivalence-ca-cmfr} The categories $\mathbb{CA}$ and $\mathbb{CM}^{\mathrm{fr}}$ are equivalent, through the functors  $(\_)^{\mathrm{cm}}$ and $(\_)^{\mathrm{ca}}$  (the former being corestricted  to, and the latter being restricted to $\mathbb{CM}^{\mathrm{fr}}$, respectively). This equivalence restricts to an equivalence between $\mathbb{CA}^{\mathrm{fd}}$ and  $\mathbb{CM}^{\mathrm{fdr}}$.
\end{theorem}
\begin{proof} The second part of the statement is an immediate consequence of Propositions \ref{th:cm-functor-correct}(2) and  \ref{th:ca-functor-correct}(2).
We now prove the first part. We first note that $\mathbf C^{\mathrm{cm}}$ is finitely ranked: this follows readily from its definition (cf. Example \ref{abstract-concrete-trace}.) Hence, $(\_)^{\mathrm{cm}}$ can be viewed as a functor from $\mathbb{CA}$ to $\mathbb{CM}^{\mathrm{fr}}$.  

If $\mathbf M$ is a finitely ranked cm-monoid, i.e.,  $\mathbf M=  \mathbf M_{|\omega}$,  we have   $\mathbf M\cong(\mathbf M^{\mathrm{ca}})^{\mathrm{cm}}$ by Lemma \ref{prop:iso}. Conversely, we have already observed that
$\mathbf C \cong (\mathbf C^{\mathrm{cm}})^{\mathrm{ca}}$ for any clone algebra $\mathbf C$. 
\end{proof}

\begin{corollary} \label{equivalence-ac-cmfdr}
  The category  $\mathbb{ACL}$ of abstract clones and 
the category $\mathbb{CM}^{\mathrm{fdr}}$ are equivalent.
\end{corollary}
\begin{proof} 
By Theorems \ref{prop:ca-ac-adjunction} and  \ref{equivalence-ca-cmfr}.
\end{proof}

\section{Partial infinitary clone algebras}\label{sec:ca2} 
In this section, we introduce a generalisation of both $\mathsf{CA}$s and $\aleph_0$-$\mathsf{AC}$s, which we relate to extensional cm-monoids via an equivalence of categories.

\begin{definition}
 A  \emph{partial infinitary clone algebra} ($\mathsf{PICA}$) is a structure $\mathbf A= (A,q^\mathbf A,\e_n^\mathbf A, D^\mathbf A)_{n\geq 0}$ of universe $A$ such that $\e_n^\mathbf A\in A$, $D^\mathbf A\subseteq A^\omega$ and $q^\mathbf A$ is a partial operation of arity $\omega$ satisfying the following conditions:
 \begin{itemize}
 \item[(P1)] $D^\mathbf A$ is a trace on $A$;
 \item[(P2)] For all $a\in A$, $\mathrm{dom}(q^\mathbf A(a,-,\dots,-,\dots))=D^\mathbf A$;
  \item[(P3)]  For all $y,z\in D^\mathbf A$, we have $(q^\mathbf A(y_0, z), \dots,q^\mathbf A(y_n, z),\dots)\in D^\mathbf A$;
   \item[(P4)] $(\e_0^\mathbf A,\e_1^\mathbf A,\dots,\e_n^\mathbf A,\dots)\in D^\mathbf A$.
 \end{itemize}
 For all $a\in A$ and $y,z\in D^\mathbf A$, the following identities hold:
  \begin{itemize}
\item[(I1)] $q^\mathbf A(\e_i^\mathbf A, y)=y_i$;
\item[(I2)] $q^\mathbf A(a,\e_0^\mathbf A,\e_1^\mathbf A,\dots,\e_n^\mathbf A,\dots)=a$;
 \item[(I3)] $q^\mathbf A(q^\mathbf A(a, y), z)=q^\mathbf A(a,q^\mathbf A(y_0, z), \dots,q^\mathbf A(y_n, z),\dots)$.
\end{itemize}
\end{definition}

A  function $f:A\to B$ is a homomorphism from a $\mathsf{PICA}$ $\mathbf A$ to a $\mathsf{PICA}$ $\mathbf B$ if the following conditions hold: (a) $f(\e_n^\mathbf A)=\e_n^\mathbf B$; (b) $s\in D^\mathbf A \Rightarrow f^\omega(s)\in  D^\mathbf B$; (c) $s\in D^\mathbf A \Rightarrow f(q^\mathbf A(a, s))=q^\mathbf B(f(a), f^\omega(s))$.
The category of $\mathsf{PICA}$s will be denoted by $\mathbb{PICA}$.

\medskip

The most important examples of $\mathsf{PICA}$s are the following.
\begin{itemize}

\item ($D^\mathbf A= A^\omega$):  These $\mathsf{PICA}$s are exactly the abstract $\aleph_0$-clones (see Section \ref{sec:neu}).

\item ($D^\mathbf A= [1]_\equiv$), where $1=(\e_0^\mathbf A,\e_1^\mathbf A,\dots,\e_n^\mathbf A,\dots)$. In this case we recover, up to equivalence, the category of clone algebras, as shown below.
\end{itemize}
\begin{lemma}\label{ca-as-pica}
   The category of clone algebras and the full subcategory of $\mathbb{PICA}$
  formed by all $\mathsf{PICA}$s $\mathbf A$ such that $D^\mathbf A= [1]_\equiv$
  are equivalent. 
\end{lemma}

\begin{proof} 
 The operator $q$ and the operators $q_n$ are interdefinable as follows: if $q$ is given, then we set
  $q_n(a,b_0,\ldots,b_{n-1})=q(a,b_0,\ldots,b_{n-1},\e_n^\mathbf A,\ldots)$, and if
  the $q_n$s are given, then we set $q(a,b_0,\ldots,b_n,\ldots)=q_k(a,b_0,\ldots,b_{k-1})$ whenever  $b_i=\e_i^\mathbf A$ for all $i\geq k$. \end{proof}

Given a $\mathsf{PICA}$ $\mathbf A=(A,q^\mathbf A,\e_i^\mathbf A,D^\mathbf A)$ and an element $a\in A$,
we say that $a$ \emph{is $\omega$-finite-dimensional} if there exists $n$ such that 
$q^\mathbf A(a,z[\e_0,\dots,\e_{n-1}])= a$ for every $z\in D^\mathbf A$. We denote by $\mathbf A_{\omega\mathrm{fin}}$ the set of $\omega$-finite-dimensional elements of a $\mathsf{PICA}$ $\mathbf A$.
The $\mathsf{PICA}$ $\mathbf A$ is \emph{$\omega$-finite-dimensional} if every element of $A$ is such.

\medskip 
We now show how to associate a cm-monoid with a partial infinitary clone algebra.

\begin{definition}
Let $\mathbf A=(A,q^\mathbf A, \e_n^\mathbf A, D^\mathbf A)$ be a $\mathsf{PICA}$.
We define an algebra $\mathbf A^{\mathrm{ecm}}=(\mathbf A^{\mathrm{ecm}}_0,\mathbf A^{\mathrm{ecm}}_1)$,
where $\mathbf A^{\mathrm{ecm}}_1$ is the canonical merge algebra of universe $D^\mathbf A\subseteq A^\omega$ (see Lemma \ref{lem:tra}). The algebra 
$\mathbf A^{\mathrm{ecm}}_0=(D^\mathbf A,\cdot^{\mathrm{ecm}},1^{\mathrm{ecm}})$ is defined  as follows, for every $ a,  b\in D^\mathbf A$: 
\begin{equation} \label{product-from-pica} b\cdot^{\mathrm{ecm}} a = (q^\mathbf A(b_0, a),\dots,q^\mathbf A(b_n, a),\dots);\quad 1^{\mathrm{ecm}}=(\e_0^\mathbf A,\dots,\e_n^\mathbf A,\dots).
\end{equation}
\end{definition}
Note that, by (P3), the operation $\cdot^{\mathrm{ecm}}$ is well-defined.

\begin{proposition}\label{thm:pic-rep}  Let $\mathbf A=(A,q^\mathbf A, \e_n^\mathbf A, D^\mathbf A)$ be a $\mathsf{PICA}$. Then we have:
\begin{enumerate}
\item The algebra $\mathbf A^{\mathrm{ecm}}$ is an extensional  cm-monoid.
\item  If $\mathbf A$ is $\omega$-finite-dimensional, then the cm-monoid $\mathbf A^{\mathrm{ecm}}$ is $\omega$-finite-dimensional. 
\end{enumerate}

\end{proposition}

\begin{proof} (1) We start by proving  that $(\mathbf A^{\mathrm{ecm}})_0$ is a monoid and that the right distributivity holds.
Let $x,y,z\in D^\mathbf A$.
\begin{enumerate}
\item[(i)]  $( x \cdot 1)=(q(x_0,1),\dots,q(x_n,1),\dots)=_{(I2)}  x$.
\item[(ii)] $(1 \cdot  x)=(q(\e_0, x),\dots,q(\e_n, x),\dots)=_{(I1)}  x$.
\item[(iii)] 
$ ( x \cdot  y)\cdot z  
   =  (q(x_i, y):i\in\omega) \cdot  z 
    =   (q(q (x_i, y), z):i\in\omega) 
     =_{(I3)}  (q(x_i,(q(y_j, z):j\in\omega)) :i\in\omega) 
        =  (q(x_i, y\cdot  z) :i\in\omega) 
       =   x \cdot ( y\cdot  z)  $.
\item[(iv)] $(x\star_n y)\cdot z=(q((x\star_n y)_0, z),\dots,q((x\star_n y)_n, z),\dots)$\\
$=(q (x_0, z),\dots,q (x_{n-1}, z),q (y_n, z)\dots) = (x\cdot z) \star_n( y \cdot z)$.
\end{enumerate}
$\mathbf A^{\mathrm{ecm}}$ satisfies (LS), because $\bar\sigma(x)=(x_{\sigma_i}:i\in\omega)=(\dots,q (\e_{\sigma_i},x),\dots)=(\e_{\sigma_i}:i\in\omega)\cdot x=\bar\sigma(1)\cdot x$. 

(2) We need to prove that every $a\in D^\mathbf A$ is $\omega$-finite-dimensional in the cm-monoid $\mathbf A^{\mathrm{ecm}}$. This means that, for every $n$, there exists $m$ such that
 $(a\cdot^{\mathrm{ecm}} (1\star_m^{\mathrm{ecm}} b))\star_n^{\mathrm{ecm}} a = a\quad \text{for every $b\in D^\mathbf A$}
$.
This is equivalent to show that, for every $0\leq i<n$,
$(a\cdot^{\mathrm{ecm}} (1\star_m^{\mathrm{ecm}} b))_i=q^\mathbf A(a_i,b[\e_0,\dots,\e_{m-1}])=a_i$ for some $m$. 
The existence of such an $m$ is guaranteed by the fact that $a_i$ ($0\leq i<n$) is $\omega$-finite-dimensional in the $\mathsf{PICA}$ $\mathbf A$.
\end{proof}

The cm-monoid $\mathbf A^{\mathrm{ecm}}$ is also faithful and noncommutative.

The correspondence $(\_)^{\mathrm{ecm}}$ mapping a $\mathsf{PICA}$ $\mathbf A$ into the cm-monoid $\mathbf A^{\mathrm{ecm}}$ establishes that $(\_)^{\mathrm{ecm}}$ is functorial on objects. To see that it is functorial on arrows, for every homomorphism $f:\mathbf A\to \mathbf B$ of $\mathsf{PICA}$s, we define
$f^{\mathrm{ecm}}=(f^\omega)_{|D^\mathbf A}: D^{\mathbf A}\to  D^{\mathbf B} $. It is easy to see that  $f^{\mathrm{ecm}}$ is well defined
and that it is a homomorphism of cm-monoids.

\begin{lemma}\label{lem:funct}
  The correspondence $(\_)^{\mathrm{ecm}}$ is a functor from
  $\mathbb{PICA}$ to $\mathbb{CM}$, preserving $\omega$-finite dimensionality.
\end{lemma}

We now go in the other direction, from extensional cm-monoids to partial infinitary clone algebras.

Let $\mathbf M$ be an extensional cm-monoid. We define a  structure $\mathbf M^{\mathrm{pica}}=(M_{|1},q^{\mathbf M},\e_k^{\mathbf M}, D^\mathbf M)$ 
on the set of elements of $M$ of rank $\leq 1$ as follows: 
 \begin{itemize}
\item[(i)] $\e_k^{\mathbf M} = 1[k]$.
\item[(ii)] $D^\mathbf M= M_{\sq}$.
\item[(iii)] For every $a\in M_{|1}$, $\mathrm{dom}(q^{\mathbf M}(a,-,\dots,-,\dots))= M_{\sq}$ and 
$q^{\mathbf M}(a,b[0],\dots,b[n],\dots)=(a\cdot b)_{\pref 1}$.
\end{itemize}

\begin{proposition} Let $\mathbf M$ be an extensional cm-monoid. Then we have:
\begin{enumerate}
\item The structure $\mathbf M^{\mathrm{pica}}$ is a $\mathsf{PICA}$.
\item If $\mathbf M$ is $\omega$-finite-dimensional, then so is $\mathbf M^{\mathrm{pica}}$.
\end{enumerate}
\end{proposition}

\begin{proof} We first prove that $\mathbf M^{\mathrm{pica}}$ is a $\mathsf{PICA}$.
\begin{enumerate}

\item[(P1)] By Lemma \ref{lem:coto}, $\mathbf M$ is isomorphic to a canonical cm-monoid of universe  $M_{\sq}$. By Lemma \ref{lem:tra}  $M_{\sq}$ is a trace on 
$M_{|1}$.

\item[(P2)] By definition of $q^{\mathbf M}$.

\item[(P3)] 
Let $a,b\in M$.  Recall that $b_\sq$ is a shorthand for $(b[0],b[1],\dots)$.
We show that\\ $(q(a[0],b_\sq),\dots,q(a[n],b_\sq),\dots)\in M_{\sq}$.
Indeed, $q(a[n],b_\sq)= (a[n]\cdot b)_{\pref 1}=_{L.\ref{lem:[k]bis}(1)} (a\cdot b)[n]$.
Then we conclude that $(q(a[0],b_\sq),\dots,q(a[n],b_\sq),\dots) = (a\cdot b)_\sq \in M_{\sq}$.

\item[(I1)]  $q(\e_k,b[0],\dots,b[n],\dots)=(1[k] \cdot b)_{\pref 1} =(\bar\tau^k_0(1)_{\pref 1}\cdot b)_{\pref 1} =_{L. \ref{lem:emme2}(1)} (\bar\tau^k_0(1)\cdot b)_{\pref 1} =b[k]$ by Lemma
  \ref{lem:[k]bis}(2).

\item[(I2)] $q(a,\e_0,\dots,\e_n,\dots)= (a\cdot 1)_{\pref 1}=a_{\pref 1}=a$, because $\e_n=1[n]$
  and by hypothesis $a_{ \pref 1}=a$.

\item[(I3)] $q(q(a, b_\sq),  c_\sq)) = ((a\cdot b)_{\pref 1} \cdot c)_{\pref 1}=_{L. \ref{lem:emme2}(1)} (a\cdot b \cdot c)_{\pref 1}= (a\cdot ( b \cdot  c))_{\pref 1}=$\\
$  q(a,(b\cdot c)[0],(b\cdot c)[1],\ldots)=
q(a,(b[0]\cdot c)_{\pref 1},(b[1]\cdot c)_{\pref 1},\ldots)=$\\
$  q(a,q(b [0], c_\sq),q(b [1],  c_\sq),\dots)$, because 
$(b \cdot c)[n] = (b[n] \cdot c)_{\pref 1} = q(b[n], c_\sq)$. 
\end{enumerate}
We now prove the second part of the statement. Let 
$a\in M_{|1}$. Since $a$ is $\omega$-finite-dimensional, there exists $m$ such that
$(a\cdot (1\star_m z))\star_1 a=a$ for all $z\in M$. For such an $m$, and for every $z\in M$ we have:
$q^{\mathbf M}(a,z_\sq[1[0],\dots,1[m-1]])=q^{\mathbf M}(a,1[0],\dots,1[m-1],z[m],z[m+1],\dots)=(a\cdot (1\star_m z))\star_1 1 = 
(a\cdot (1\star_m z))\star_1 a=a$,
showing that $a$ is $\omega$-finite-dimensional in $\mathbf M^{\mathrm{pica}}$.
\end{proof}

Let  $\mathbb{ECM}$ be the full subcategory of $\mathbb{CM}$ whose objects are the extensional cm-monoids. 
The correspondence $(\_)^{\mathrm{pica}}$, mapping an extensional
cm-monoid  $\mathbf M$ into the  $\mathsf{PICA}$  $\mathbf M^{\mathrm{pica}}$ is the object part of a functor which is defined as follows on arrows. Let $f:\mathbf M\to \mathbf N$ be a homomorphism of cm-monoids. Then we define
$f^{\mathrm{pica}}=f_{|(M_{|1}) }: M_{|1}\to  N_{|1} $. It is easy to see that  $f^{\mathrm{pica}}$ is well defined
and that it is a homomorphism of $\mathsf{PICA}$s.

\begin{lemma}\label{lem:funct1}
  The correspondence $(\_)^{\mathrm{pica}}$ is a functor from
  $\mathbb{ECM}$ to $\mathbb{PICA}$, preserving $\omega$-finite dimensionality.
\end{lemma}

\begin{theorem}\label{prop:funct2}
The categories   $\mathbb{ECM}$ and $\mathbb{PICA}$  are equivalent, as well as the respective full subcategories of $\omega$-finite-dimensional structures.
\end{theorem}
\begin{proof} In this proof we denote by $F$ the functor  $(\_)^{\mathsf{pica}}$
    and  by $G$ the functor  $(\_)^{\mathsf{cm}}$.
  We define two natural isomorphisms $\epsilon:F\circ G\to \mathrm{Id}_{\mathbb{PICA}} $  and 
  $\eta: \mathrm{Id}_{\mathbb{ECM}} \to G\circ F$. For every extensional cm-monoid
  $\mathbf M$, $\eta_{\mathbf M}:\mathbf M\to G(F(\mathbf M))$ is the isomorphism defined by $x\mapsto x_\sq$. For every $\mathsf{PICA}$ $\mathbf A$,
  $\epsilon_A: F(G(\mathbf A))\to \mathbf A$  is the isomorphism defined by
  $(a,\e^\mathbf A_1,\ldots,\e^\mathbf A_n,\ldots)\mapsto a$.
\end{proof}

\begin{remark}
We can  obtain Theorem \ref{equivalence-ca-cmfr} as a corollary of Theorem \ref{prop:funct2}, thus providing a link between the constructions of Sections 
\ref{sec:ca1} and \ref{sec:ca2}. We only sketch the arguments, leaving the details to the reader.  One proves that, modulo the identification of $\mathbb{CA}$ with the full subcategory of $\mathbb{PICA}$ spelled out in Lemma \ref{ca-as-pica}, (i) the functor $(\_)^{\mathrm{ecm}}$ restricted to $\mathbb{CA}$ ``is''  $(\_)^{\mathrm{cm}}$, and (ii) the restriction of $(\_)^{\mathrm{pica}}$ to finitely ranked  cm-monoids  ``is'' 
$(\_)^{\mathrm{ca}}$. The following properties are implicit in these two claims:
\begin{itemize}
\item Finitely-ranked cm-monoids are extensional by   Lemma \ref{lem:coto}(2).
\item An extensional cm-monoid $M$ is finitely ranked  if and only if $M_{\sq}$ is equal to the  trace $[1]_\equiv$. This claim is proved as follows. The inclusion $[1]_\equiv\subseteq M_{\sq}$ follows from Lemma \ref{lem:bar}  (without using the extensionality assumption). If $M$ is finitely ranked, then the inclusion $M_{\sq}\subseteq [1]_\equiv$ follows from Lemma \ref{lem:[k]}(ii). Conversely, if $M_{\sq}\subseteq [1]_\equiv$, then for any $x\in M$ we have $x[n]=1[n]$ for all sufficiently large $n$, from which one concludes $x=x_{<n}$ by extensionality.
\end{itemize}
\end{remark}

\bibliographystyle{plain}
\bibliography{main}

\end{document}